\documentclass[12pt,leqno]{article}
\usepackage{amssymb,amsfonts,amsmath,amsthm,amscd,mathrsfs}
\usepackage{MnSymbol}
\usepackage{PSFigure,psfrag}
\usepackage{graphicx}
\usepackage{centernot}
\usepackage{stmaryrd}

\def\IR{{\mathbb R}}

\def\IT{{\mathbb T}}

\def\IZ{{\mathbb Z}}
\def\IV{{\mathbb V}}

\def\n{\noindent}

\def\dis{\displaystyle}

\def\fr{\mbox{\footnotesize $\dis\frac{1}{2}$}}

\def\ov{\overline}
\def\ve{\varepsilon}
\def\f{\footnotesize}

\def\r{\rightarrow}

\def\wh{\widehat}
\def\wt{\widetilde}

\def\cH{{\cal H}}

\def\cF{{\cal F}}

\def\cV{{\cal V}}

\dimendef\dimen=0

\newtheorem{theorem}{Theorem}[section]
\newtheorem{lemma}[theorem]{Lemma}
\newtheorem{corollary}[theorem]{Corollary}
\newtheorem{proposition}[theorem]{Proposition}
\newtheorem{remark}[theorem]{Remark}

\begin{document}

\noindent
~

\bigskip
\begin{center}
{\bf AN EFFECTIVE ENERGY-ENSTROPHY DIFFUSION PROCESS \\ WITH A CONDENSATION BOUND}
\end{center}

\begin{center}
Alain-Sol Sznitman$^1$ and Klaus Widmayer$^2$
\\[2ex]
Preliminary Draft
\end{center}

\begin{abstract}
We use Gaussian measure on $\IR^N$ to define the coefficients of an elliptic diffusion and show that it lives in an open cone of $\IR^2$. One component represents enstrophy and the other energy. We establish the existence and uniqueness of a stationary distribution for this diffusion. Owing to the special properties of the coefficients of this diffusion, we derive a condensation bound, which controls the distance to $1$ of the ratio of the expected energy to the expected enstrophy (this ratio is at most  $1$ with our normalization). In a companion article, as a ``proof of concept'', we show that the diffusion constructed in this work is the inviscid limit of the laws of the ``enstrophy-energy'' process of a stationary $N$-dimensional Galerkin-Navier-Stokes type evolution with Brownian forcing and random stirring (the strength of which can be made to go to zero in the inviscid limit, and which plays the role of a regularization).

\end{abstract}

\vfill
\hfill July 2026
\vfill
\noindent
{\footnotesize
-------------------------------- \\
$^1\,$Department of Mathematics, ETH Zurich, CH-8092 Zurich, Switzerland\\
$^2\,$Institute of Mathematics, University of Zurich, CH-8057 Zurich, Switzerland and Faculty of 

\vspace{-0.5ex}
\noindent
~~\!Mathematics, University of Vienna, A-1090 Vienna, Austria
}

%\vfill

%~
\newpage
\thispagestyle{empty}
~

\newpage
\setcounter{page}{1}

\setcounter{section}{-1}
\section{Introduction}

The incompressible two-dimensional Navier-Stokes equation with Brownian forcing on a square or a thin torus is a classical infinite dimensional Markovian evolution, see \cite{KuksShir12}. There is an extensive theory concerning the existence and the uniqueness of a stationary distribution for these Markov processes, see \cite{HairMatt06}, \cite{HairMatt11}, \cite{KuksShir12}. However, the nature of these stationary distributions remains, except for some simple and rather unrepresentative cases, ill understood. In particular, when an inviscid limit, with an adequate balancing of the forcing, is performed, little is known about their limit behavior, see \cite{GlatSverVico15}, \cite{Kuks08}, and Chapter 5 of \cite{KuksShir12}. This feature persists when considering finite dimensional evolutions corresponding to Galerkin-Navier-Stokes. While chaotic behavior of the dynamics has been established, see \cite{BedrBlumPuns22}, \cite{BedrPuns24}, straightforward questions remain unanswered: existing lower bounds on the ratio between the expectation of the energy and the expectation of the vorticity (the latter being fixed in this limit procedure) are so far blind to the limit procedure, the behavior of the variance of the enstrophy is poorly understood, and the same feature holds concerning the expected trend of the stationary distributions to condense on low Fourier modes under suitable forcings, see \cite{BeckCoopLordSpil20}, \cite{BoucSimo09}, (4.2) in \cite{LinkHohmEckh20},  p.~127 of \cite{Naza11}.

\medskip
In the present work, motivated by these questions, we introduce and study a stationary diffusion process in an open cone of $\IR^2$. We view the components of the process as describing enstrophy and energy. We establish a condensation bound for this diffusion process. It controls in a quantitative fashion the distance to $1$ of the ratio of the expected energy to the expected enstrophy (this ratio is at most $1$ with our normalization). The coefficients of the two-dimensional diffusion process involve the conditional expectations for a Gaussian measure on $\IR^N$ of the squares of the coordinate functions given the value of two quadratic forms on $\IR^N$ respectively corresponding to the enstrophy and the energy. These conditional expectations satisfy a remarkable monotonicity property, which plays an important role in the derivation of the condensation bound. In the companion article \cite{SzniWidm26b}, as a ``proof of concept'', we show that the two-dimensional stochastic process studied here is the inviscid limit of the laws of the ``enstrophy-energy'' process of a stationary $N$-dimensional Galerkin-Navier-Stokes type evolution with Brownian forcing and random stirring (the strength of which can be made to go to zero in the inviscid limit).

\medskip
We will now discuss the results in more detail.

\medskip
We refer to Section 1 for precise assumptions. We consider on $\IR^N$ with $N = 2n$, and $n \ge 4$, two quadratic forms $|x|^2 = \Sigma_\ell \,x_\ell^2$ and $|x|^2_{-1} = \Sigma_\ell \, x_\ell^2/\lambda_\ell$, where $\lambda_{2i} = \lambda_{2i-1}$ (denoted by $\mu_i), 1 \le i \le n$, is increasing and $\mu_1 = 1$. The first quadratic form corresponds to (twice) the enstrophy and the second to (twice) the energy in the terminology of the paragraph above. We refer to Remark \ref{rem1.1} for the typical example we have in mind. We also have a Gaussian measure $\mu$ on $\IR^N$ under which the coordinate functions $x_\ell, 1 \le \ell \le N$, are independent centered Gaussian variables with same variance equal to $a/2$, with $a > 0$. Further, we consider the two-dimensional cone
\begin{equation}\label{0.1}
C = \{w = (u,v) \in \IR^2; \,0 \le v \le u\le \lambda_N v\},
\end{equation}
and its interior $\stackrel{_\circ}{C}$. We construct in Section 2, see also Appendix B, good versions $q_\ell(u,v)$ of the conditional expectation of $x^2_\ell$ given $|x|^2 = u$, and $|x|^2_{-1} = v$, such that
\begin{equation}\label{0.2}
q_\ell (|x|^2, |x|^2_{-1}) = E^\mu\big[x_\ell^2\,\big| \,|x|^2, |x|^2_{-1}\big], \, \mbox{$\mu$-a.s. for $1 \le \ell \le N$}.
\end{equation}
The functions $q_\ell$ are Lipschitz in $C$, positive in $\stackrel{_\circ}{C}$, homogeneous of degree $1$, and on each subsector $\{(u,v) = C$; $0 \le \mu_{i-1} \,v \le u \le \mu_i v\}$, with $2 \le i \le n$, they coincide with a rational function (dependent on $i$).

\medskip
With the functions $q_\ell$, and a sequence $\delta_\ell \in (-1,0]$, $1 \le \ell \le N$, such that $\delta_{2i} = \delta_{2i-1}$, for $1 \le i \le n$ (see also the observation below (\ref{3.6}) about this last condition), we construct the differential operator on $\stackrel{_\circ}{C}$:
\begin{equation}\label{0.3}
\begin{split}
\wt{A} \psi  = &\; 2a\Big\{\dis\sum_\ell \lambda_\ell ( 1 + \delta_\ell) \, q_\ell \,\partial^2_u \psi + 2 \dis\sum_\ell (1 + \delta_\ell) \,q_\ell\,\partial^2_{u,v} \psi + \dis\sum_\ell \, \mbox{\f $\dis\frac{1}{\lambda_\ell}$} (1 + \delta_\ell) \,q_\ell \,\partial_v^2 \psi\Big\}
\\
& + \Big(a \dis\sum_\ell \,\lambda_\ell (1 + \delta_\ell) - 2 \dis\sum_\ell \lambda_\ell \, q_\ell\Big) \, \partial_u \psi +  \Big(a \dis\sum_\ell \,(1 + \delta_\ell) - 2 \dis\sum_\ell q_\ell\Big) \, \partial_v\psi .
\end{split}
\end{equation}

\n
We show in Section 3 that (\ref{0.3}) gives rise to an elliptic second order differential operator with Lipschitz coefficients in $\stackrel{_\circ}{C}$, and a well-posed martingale problem, see Theorem \ref{theo3.2}. We prove in Section 4 that this diffusion process has a unique stationary distribution, see Theorem \ref{theo4.1}, and we provide a martingale problem characterization for the law $\wt{P}$ of the stationary process, see Proposition \ref{prop4.2} (which is useful in the context of the convergence in law that we prove in \cite{SzniWidm26b}).

\medskip
The stationary distribution mentioned above is typically unknown, but when all $\delta_\ell$ coincide it is explicit. For instance, when all $\delta_\ell = 0$, it coincides with the image measure of $\mu$ under the map: $x \longrightarrow (|x|^2, |x|^2_{-1})$, see Remark \ref{rem4.5}.

\medskip
In the companion article \cite{SzniWidm26b}, as a ``proof of concept'', we establish that $\wt{P}$ mentioned above, is the limit in law of the process $(|X^\ve_t |^2, |X^\ve_t|^2_{-1})_{t \ge 0}$, as $\ve \rightarrow 0$, when $X^\ve_t, t \ge 0$, is the stationary solution of the Stratonovich differential equation on $\IR^N$:
\begin{equation}\label{0.4}
dX^\ve_t = \big( -\Lambda X^\ve_t + \mbox{\f $\dis\frac{1}{\ve}$} \, b(X^\ve_t)\big) \,dt + \dis\sum_\ell \big(\lambda_\ell \, a(1 + \delta_\ell)\big)^{\frac{1}{2}} e_\ell \,\, d \beta_\ell(t) + \Big(\mbox{\f $\dis\frac{\kappa}{\ve}$} \Big)^{\frac{1}{2}} \dis\sum^M_1 Z_m(X^\ve_t) \, \circ d \wt{\beta}_m(t),
\end{equation}
where $(\Lambda x)_\ell = \lambda_\ell \, x_\ell$, for $x \in \IR^N$, $1 \le \ell \le N$, $(e_\ell)_{1 \le \ell \le N}$ is the canonical basis of $\IR^N$, $\beta_\ell$, $1 \le \ell \le N$, $\wt{\beta}_m, 1 \le m \le M$, are independent Brownian motions, and $\kappa > 0$ measures the strength of the stirring and can be chosen arbitrarily small (and the limit law $\wt{P}$ does not depend on $\kappa$). Moreover, $b(x)_\ell$ is for each $\ell$ a quadratic form in the coordinates of $x$, such that ${\rm div} \, b = 0$, $\langle x, b(x) \rangle = 0$, $\langle x, b(x)\rangle_{-1} = 0$ on $\IR^N$ (with $\langle \cdot, \cdot \rangle$ and $\langle \cdot, \cdot \rangle_{-1}$ the scalar products attached to $| \cdot |$ and $| \cdot |_{-1}$). As for the stirring vector fields $Z_m$, $1 \le m \le M$, they are constructed in \cite{SzniWidm26b}, and satisfy ${\rm div} \,Z_m = 0$, $\langle x, Z_m(x)\rangle = 0$, $\langle x, Z_m(x)\rangle_{-1} = 0$ on $\IR^N$. They play the role of a regularization. 

\medskip
As mentioned above, the convergence to the same limit law $\wt{P}$ holds for all values of $\kappa > 0$ in (\ref{0.4}), and hence the convergence in law to $\wt{P}$ holds as well for some (non-explicit) function $\kappa_\ve > 0$ tending to $0$ with $\ve$. When $\kappa = 0$,  the equation corresponding to (\ref{0.4}) covers by suitable choices of $b(x)$ some instances of a Galerkin-Navier-Stokes equation with Brownian forcing, see Remark 1.1 of \cite{SzniWidm26b}. It would be of substantial interest to gain information on the speed at which one can let $\kappa_\ve$ tend to $0$, i.e. how much regularization is needed, for the above mentioned convergence to $\wt{P}$ to hold.

\medskip
In the present article, we obtain a remarkable condensation bound for the law $\wt{P}$ of the stationary diffusion in the open cone $\stackrel{_\circ}{C}$. It can be read in the perspective of the inviscid limit in law statement mentioned above (\ref{0.4}). We introduce two positive numbers related to the Brownian forcing:
\begin{equation}\label{0.5}
B_0 = a \sum_\ell \, (1 + \delta_\ell) < B_1 = a \sum_\ell \, \lambda_\ell (1 + \delta_\ell),
\end{equation}    
and we denote by $(U_0, V_0)$ the location at time $0$ of the stationary diffusion on $\stackrel{_\circ}{C}$ governed by $\wt{P}$ (so that $U_0 > V_0 > 0, \wt{P}$-a.s.). We show in Theorem \ref{theo5.1} a stronger form of the following inequality:
\begin{equation}\label{0.6}
\begin{array}{l}
\mbox{for any $\ell_0$ in $\{3, \dots, N\}$, $2 \wt{E} [ U_0 - V_0] \le \dis\frac{B_1 - B_0}{\lambda_{\ell_0} - 1} + \dis\frac{\lambda_3}{\lambda_3 - 1} \; \dis\frac{\ell_0}{N- \ell_0} \; B_0$}, 
\\[1ex]
\mbox{with $B_0 = 2 \wt{E} [U_0]$}
\end{array}
\end{equation}
(and $\wt{E} [ \cdot]$ the $\wt{P}$-expectation).

\bigskip
In particular, when $\ell_0$ can be chosen so that $B_1 / B_0$ is much smaller than $\lambda_{\ell_0}$ and $\ell_0$ is much smaller than $N$, then $\wt{E} [U_0 - V_0] \,/ \, \wt{E}[U_0]$ is close to $0$, corresponding to ``condensation''. The value $B_1 / B_0$ can be viewed as a kind of effective spectral value for the Brownian forcing, and the bound (\ref{0.6}) is thus informative when this effective spectral value of the forcing is much smaller than $\lambda_N$, see also Remark 6.2 of \cite{SzniWidm26b}.

\medskip
As it turns out, the proof of the condensation bound in Theorem \ref{theo5.1} relies on the surprizing \textit{monotonicity property} of the functions $q_\ell$, as $\ell$ varies in $\{3, \dots, N\}$, see Theorem \ref{theo2.3}. In particular, it implies that the functions $q_\ell$ with $\ell$ small compared to $N$ have size at most of order $1/N$. This feature links the Gaussian measure $\mu$ to the condensation effect for the possibly highly non-Gaussian measure corresponding to the stationary distribution of the diffusion in $\stackrel{_\circ}{C}$, when $B_1 / B_0$ is much smaller than $\lambda_N$.

\medskip
As an aside, ``condensation'' is known to occur under the Gaussian measure, see Section 3.1 of \cite{BoucCorv10}. In essence, fixing an infinite sequence $\lambda_\ell$ tending to infinity, in a large $N$ regime, under the measure $\mu$, when $0 < v < u$ are given, and one ``forces $|x|^2/N$ and $|x|^2_{-1}/N$ to respectively take values close to $u$ and $v$'', then $(x^2_1 + x^2_2)/N$ is close to $v$ (and the ratio of the sum of the remaining $x^2_\ell$ with $N$ is close to $u-v$). In the present article, we see that Gaussian measure also induces condensation for the stationary distribution of some of the diffusions that we consider; but this effect now happens through the properties of the functions $q_\ell$, notably Theorem \ref{theo2.3}.

\medskip
Let us finally describe the organization of the article. 

\medskip
Section 1 presents the set-up. Section 2 collects some properties of the functions $q_\ell$, imported in part from Appendix B, and proves the important monotonicity property in Theorem \ref{theo2.3}. Section 3 constructs the diffusion process in Theorem \ref{theo3.2}, and establishes a Lyapunov-Foster condition in Proposition \ref{prop3.3}. Section 4 is concerned with the stationary distribution of the diffusion process, and collects some of its properties. Section 5 is devoted to the ``condensation bound'', see Theorem \ref{theo5.1}. Appendix A contains some probability density and volume calculations. Appendix B establishes some of the properties of the functions $q_\ell$.

\bigskip\bigskip\n
{\bf Acknowledgment:} The authors wish to thank Tobias Rohner for extensive simulations of the incompressible Navier-Stokes equation with Brownian forcings on a two-dimensional torus.

\bigskip
\section{The set-up}
\setcounter{equation}{0}

In this section we introduce some notation and various objects that we will consider in what follows.

\medskip
We have an even integer
\begin{equation}\label{1.1}
N = 2 n, \; \mbox{with $n \ge 4$}
\end{equation}
(in Appendix A we will only assume $n \ge 3$), as well as the positive $\lambda_\ell, 1 \le \ell \le N$, such that
\begin{equation}\label{1.2}
1 = \lambda_1 = \lambda_2 < \lambda_3 = \lambda_4 < \dots < \lambda_{2 i -1} = \lambda_{2i} < \dots < \lambda_{N-1} = \lambda_N,
\end{equation}
and the scalar products on $\IR^N$ associated to the square norms:
\begin{equation}\label{1.3}
|x|^2 = \sum_\ell \, x_\ell^2  \ge |x|^2_{-1} = \sum \; \dis\frac{x^2_\ell}{\lambda_\ell}, \; \mbox{for $x \in \IR^N$}
\end{equation}
(the corresponding scalar products are denoted by $\langle \cdot, \cdot \rangle$ and $\langle \cdot, \cdot\rangle_{-1}$). We sometimes refer to the coordinates $x_\ell$ of $x$ as $\textit{modes}$, $x_1$ and $x_2$ being the $\textit{lowest modes}$, see also Remark \ref{rem1.1} below.  While (\ref{1.2}) is somewhat restrictive in view of the typical example we have in mind in Remark \ref{rem1.1}, it offers a convenient setting to develop the results of Section 2: it allows us to avoid having to right away handle some possible degeneracies of some of the $\Delta_i$, $1 < i \le n$ in (\ref{2.1}). We often use the notation
\begin{equation}\label{1.4}
\mbox{$\mu_i = \lambda_{2i}$, for $1 \le i \le n$ (and $n$ as in (\ref{1.1}))}.
\end{equation}
Further, we have a Gaussian measure $\mu$ on $\IR^N$ corresponding to
\begin{equation}\label{1.5}
a > 0,
\end{equation}
and 
\begin{equation}\label{1.6}
d \mu = (\pi a)^{-N/2} \exp \big\{-\dis\frac{|x|^2}{a} \big\} \;dx
\end{equation}
so that under $\mu$ the coordinates $x_\ell, 1 \le \ell \le N$, are i.i.d. centered Gaussian variables with variance~$\frac{a}{2}$.

\medskip
In addition, we consider coefficients
\begin{equation}\label{1.7}
\mbox{$\delta_\ell \in (-1, 0]$, $1 \le \ell \le N$, such that $\delta_{2i} = \delta_{2i-1}$, for $1 \le i \le n$}.
\end{equation}
They enter the definition of the diffusion operator
\begin{equation}\label{1.8}
\wt{L} = \sum_\ell \, \lambda_\ell \,\Big(\frac{a}{2} \;(1 + \delta_\ell) \, \partial^2_\ell - x_\ell \, \partial_\ell\Big)
\end{equation}
corresponding to independent Ornstein-Uhlenbeck processes for each coordinate, with respective stationary variances $\frac{a}{2} (1 + \delta_\ell)$ and speed-up factors $\lambda_\ell$.

\begin{remark}\label{rem1.1}
\rm The typical example we have in mind corresponds to a ``Galerkin projection'', when one considers a possibly thin two-dimensional torus $\IT = (\IR / 2 \pi \mu \IZ) \times (\IR / 2 \pi \IZ)$, with $0 < \mu \le 1$, and $\lambda_\ell$, $1 \le \ell \le N$, are eigenvalues of the Laplacian on the torus, $\varphi_\ell$ are $L^2(\IT)$-normalized pairwise orthogonal, and orthogonal to constants, eigenfunctions of the Laplacian attached to the $\lambda_\ell$, $1\le \ell \le N$, and the space is chosen so that (\ref{1.1}) and (\ref{1.2}) hold. Then $x = (x_\ell)$ in $\IR^N$ corresponds to the coordinates of a function $\varphi$ in the space generated by the function $(\varphi_\ell)_{1 \le \ell \le N}$, $|x|$ to the $L^2$-norm of $\varphi$, $|x|_{-1}$ to the $H^{-1}$-norm $\langle \varphi, (-\Delta)^{-1} \,\varphi \rangle^{1/2}_{L^2(\IT)}$ of $\varphi$, so that $|x|^2$ and $|x|^2_{-1}$ correspond to twice the enstrophy and the energy of the function $\varphi$. Further, in this context, $\wt{L}$ corresponds to the generator of the diffusion obtained by looking at the coordinates in the basis $(\varphi_\ell)_{1 \le \ell \le N}$ of the solution of the stochastic differential equation with Brownian forcing
\begin{equation}\label{1.9}
d \varphi_t = \Delta \varphi_t \, dt + \sum_1^N \,(\lambda_\ell \, a\big(1 + \delta_\ell)\big)^{\frac{1}{2}} e_\ell \,\,d \beta_\ell(t),
\end{equation}

\n
where $\beta_\ell$, $1 \le \ell \le N$, are independent Brownian motions, and $e_\ell$, $1 \le \ell \le N$, the canonical basis of $\IR^N$.

\medskip
In the companion article \cite{SzniWidm26b}, a strong perturbation of this linear equation with forcing will be considered. It will cover cases of a Galerkin projection on the space generated by the $\varphi_\ell$, $1 \le \ell \le N$, for the two-dimensional Navier-Stokes equation in vorticity form, with some (arbitrarily small) stirring, playing the role of a regularization, and the same Brownian forcing as in (\ref{1.9}) above. The stationary two-dimensional diffusion we construct in the present article will appear as the inviscid limit evolution of (twice) the enstrophy-energy process for the stationary process that we construct in \cite{SzniWidm26b}, see also (\ref{0.4}). \hfill $\square$
\end{remark}

We now describe the effect of $\wt{L}$ on functions solely depending on $|x|^2$ and $|x|^2_{-1}$. Namely, when $\psi (u,v)$ is a $C^2$-function on $\IR^2$, setting $g(x) = \psi(|x|^2, |x|^2_{-1})$, for $x$ in $\IR^N$, one has
\begin{equation}\label{1.10}
\begin{split}
\wt{L} g(x) = & \; 2a \Big\{ \sum_\ell \, \lambda_\ell (1 + \delta_\ell) \,x^2_\ell \, \partial_u^2 \psi + 2 \sum_\ell (1 + \delta_\ell) \, x^2_\ell \, \partial^2_{u,v} \psi + \sum_\ell \dis\frac{1}{\lambda_\ell} \;(1 + \delta_\ell) \, x_\ell^2 \, \partial^2_v \psi\Big\}
\\
& + \Big(a \sum_\ell \,\lambda_\ell(1 + \delta_\ell) - 2 \sum_\ell \,\lambda_\ell \, x_\ell^2\Big) \, \partial_u \psi + \big(a \sum_\ell (1 + \delta_\ell) - 2 \sum_\ell \, x_\ell^2\Big) \, \partial_v \psi
\end{split}
\end{equation}
(and the partial derivatives of $\psi$ are evaluated at $(|x|^2, |x|^2_{-1})$).

\medskip
The function $x \longrightarrow (|x|^2, |x|^2_{-1})$ naturally takes values in the closed cone $C$ of $\IR^2$:
\begin{equation}\label{1.11}
C = \{ w = (u,v) \in \IR^2; \, 0 \le v \le u \le \lambda_N v\} \,.
\end{equation}

\n
In essence, we will replace the $x^2_\ell$, $1 \le \ell \le N$, in (\ref{1.9}) by ``good versions'' $q_\ell (u,v)$ of the conditional expectations $E^\mu[x^2_\ell \,\big| \, |x|^2 = u, |x|^2_{-1} = v]$, and we will see in Sections 3 and 4 that one obtains a diffusion on the interior $\stackrel{_\circ}{C}$ of $C$, which has a unique stationary distribution, see Theorems \ref{theo3.2} and \ref{theo4.1}.

\section{Some properties of the functions $q_\ell$}
\setcounter{equation}{0}

In this section we collect from Appendix B some useful properties of the non-negative functions $q_\ell$ on $C$, $1 \le \ell \le N$, which provide ``good versions'' of the conditional expectations under $\mu$ of $x^2_\ell$ given $|x|^2 = u$ and $|x|^2_{-1} = v$. In Lemma \ref{lem2.1} we prove that they are Lipschitz functions, and in Theorem \ref{theo2.3} we derive a monotonicity property and an upper bound on the functions $q_\ell$, $3 \le \ell \le N$, which will play an important role in Section 5 when deriving the condensation bound in Theorem \ref{theo5.1}

\medskip
We first collect some facts from Appendix $B$, see Lemma \ref{lemB.1}, Proposition \ref{propB.2}, (\ref{B.17}). We recall the notation (\ref{1.4}) for $\mu_i$ and (\ref{1.11}) for $C$.
\begin{equation}\label{2.1}
\begin{array}{l}
\mbox{For $1 \le \ell \le N$, $q_\ell$ is a continuous non-negative function on $C$, which does not}\\
\mbox{depend on the parameter $a$ in (\ref{1.5}), homogeneous of degree $1$, positive on $\stackrel{_\circ}{C}$,}\\
\mbox{which on each open sector $\Delta_i = \{w = (u,v) \in C$; $0 < \mu_{i-1} v < u < \mu_i v\}$, $1 < i \le n$}\\
\mbox{coincides with the ratio of an homogeneous polynomial of degree $(n-1)$ in}\\
\mbox{$u,v$ with an homogeneous polynomial of degree $(n-2)$ in $u,v$.}
\end{array}
\end{equation}
Further,
\begin{align}
& q_{2i} = q_{2i-1}, \; \mbox{for $1 \le i \le n$ (and we write $\wh{q}_i = q_{2i} + q_{2i - 1})$}, \label{2.2}
\\ 
& u = \sum_1^N \,q_\ell(u,v), \, v = \sum_1^N \; \dis\frac{1}{\lambda_\ell} \;q_\ell (u,v), \; \mbox{for all $(u,v) \in C$}. \label{2.3}
\end{align}
In addition, one has the boundary values
\begin{align}
& q_\ell (u,v) = 0 \; \mbox{for $u = v$ or $u = \lambda_N v$, when $3 \le \ell \le N-2$, and $u \ge 0$}, \label{2.4}
\\[1ex]
&q_\ell (u,u) = \dis\frac{u}{2} \; \mbox{and $q_\ell \,(u, u/\lambda_N) = 0$, when $\ell = 1$ or $2$, and $u \ge 0$}, \label{2.5}
\\[1ex]
& q_\ell(u,u) = 0 \; \mbox{and $q_\ell \, (u, u/\lambda_N) = \dis\frac{u}{2}$, when $\ell = N - 1$ or $N-2$, and $u \ge 0$}, \label{2.6}
\end{align}

\n
Importantly, the $q_\ell$ (which do not depend on $a$) provide regular versions of the conditional expectation of $x^2_\ell$ under $\mu$, given $|x|^2 = u$ and $|x|^2_{-1} = v$. Namely, one has
\begin{equation}\label{2.7}
E^\mu \big[x^2_\ell \, \big| \, |x|^2, |x|^2_{-1}\big] = q_\ell (|x|^2, |x|^2_{-1}), \;\mbox{$\mu$-a.s., for $1 \le \ell \le N$.}
\end{equation}
We also refer to Remark A.2 of the companion article \cite{SzniWidm26b} for a sharper version of (\ref{2.7}). The next lemma will play an important role in the analysis of the diffusion process constructed in Section 3 and of its stationary distribution in Theorem \ref{4.1}.

\begin{lemma}\label{lem2.1}
\begin{equation}\label{2.8}
\mbox{For each $1 \le \ell \le N$, $q_\ell$ is a Lipschitz function on $C$.}
\end{equation}
\end{lemma}

\begin{proof}
Indeed, on each interval $[\mu_j, \mu_{j+1}]$, $1 \le j < n$, the rational function $r_\ell(\cdot) = q_\ell(\cdot, 1)$ cannot have in reduced form a pole in $[\mu_j, \mu_{j+1}]$, for otherwise $r_\ell$ would be unbounded in $[\mu_j, \mu_{j+1}]$, thus contradicting (\ref{2.1}). Hence, $q_\ell( \cdot, 1)$ has a bounded derivative on each interval $(\mu_j, \mu_{j+1})$. Moreover, $q_\ell(u,v) = v \, q_\ell(\frac{u}{v}, 1)$ for $(u,v)$ in $\stackrel{_\circ}{C}$, and since $u/v \le \lambda_N$,
\begin{equation}\label{2.9}
\begin{array}{l}
\partial_u \,q_\ell(u,v) = q^\prime_\ell \;\Big(\mbox{\f $\dis\frac{u}{v}$}, 1\Big)\; \mbox{and} \; \partial_v \, q_\ell (u,v) = q_\ell \; \Big(\mbox{\f $\dis\frac{u}{v}$}, 1\Big) - \mbox{\f $\dis\frac{u}{v}$} \; q^\prime_\ell \; \Big(\mbox{\f $\dis\frac{u}{v}$}, 1\Big)
\\
\mbox{are bounded functions on $\{(u,v) \in C$; $u \not= \mu_j v, 1 \le j \le n\} \stackrel{(\ref{2.1})}{=} \bigcup\limits_{1 \le i \le n} \Delta_i$}. 
\end{array}
\end{equation}
Combined with the continuity of the functions $q_\ell$, see (\ref{2.1}), the above boundedness of their gradients implies (\ref{2.8}).
\end{proof}

In the sub-sector of $C$ where $0 \le v \le u \le \mu_2 \,v$ (i.e.~$(u,v) \in \ov{\Delta}_2$, see (\ref{2.1})), one has explicit formulas for the functions $q_\ell$, see (\ref{B.27}), namely:
\begin{equation}\label{2.10}
\begin{array}{l}
q_\ell (u,v) = \mbox{\f $\dis\frac{1}{N-2}$} \; \mbox{\f $\dis\frac{u-v}{1 - \lambda^{-1}_\ell}$}\;\;\mbox{for $3 \le \ell \le N$ and}
\\[2ex]
q_\ell (u,v) = \mbox{\f $\dis\frac{v}{2}$} - \mbox{\f $ \dis\frac{1}{2} \dis\frac{1}{N-2}$}  \; \dis\sum_{m \ge 3} \; \mbox{\f $\dis\frac{u-v}{\lambda_m - 1}$} \; \mbox{for $\ell = 1,2$, when $(u,v) \in \ov{\Delta}_1$}, 
\end{array}
\end{equation}

\begin{remark}\label{rem2.10} \rm
Note that $u \le \lambda_3 \, v$ by assumption in (\ref{2.10}), so that
\begin{equation*}
q_1 (u,v) = q_2 (u,v) \ge \mbox{ \f $\dis\frac{v}{2}$} - \mbox{\f $\dis\frac{1}{2} \dis\frac{1}{N-2}$} \;  \dis\sum_{\ell \ge 3}\; \mbox{\f $\dis\frac{\lambda_3- 1}{\lambda_\ell - 1}$} \; v \ge 0.
\end{equation*}
Also, the formulas (\ref{2.10}) cannot be expected to extend to the entire cone $C$, for they do not satisfy the boundary conditions (\ref{2.4}) - (\ref{2.6}). \hfill $\square$
\end{remark}

The next proposition plays a crucial role in the proof of the condensation bound in Theorem \ref{theo5.1}. In particular, it shows that the functions $q_\ell$ of low index $\ell$ are small when $N$ is large. An important ingredient to this effect is the monotonicity property (\ref{2.11}) in the theorem below. We recall the notation $\wh{q}_i$ in (\ref{2.2}).

\begin{theorem}\label{theo2.3} {\it (Monotonicity and upper bound)}
\begin{align}
& \mbox{$(1 - \mu_i^{-1}) \; \wh{q}_i (u,v)$ is non-decreasing in $i$ for $2 \le i \le n$ and $(u,v) \in C$}. \label{2.11}
\\[1ex]
& \wh{q}_i \, (u,v) \le [(n - i + 1) (1 - \mu_i^{-1})]^{-1} (u-v), \;\mbox{for $2 \le i \le n$ and $(u,v) \in C$}. \label{2.12}
\end{align}
\end{theorem}

\begin{proof}
Note that by (\ref{2.3}) one has
\begin{equation}\label{2.13}
\dis\sum^n_2 \;(1 - \mu^{-1}_i) \; \wh{q}_i \,(u,v) = u-v, \; \mbox{for all $(u,v) \in C$}.
\end{equation}
All the summands in the left member above are non-negative and (\ref{2.12}) is an immediate consequence of the monotonicity property (\ref{2.11}). To complete the proof of Theorem \ref{theo2.3}, we thus only need to show (\ref{2.11}). Due to the continuity of the $\wh{q}_i$, it suffices to prove (\ref{2.11}) when $(u,v) \in \; \stackrel{_\circ}{C}$.

\medskip
To this end, we introduce $\wt{\Sigma} = \{2,\dots,n\}$ (we recall that $\Sigma = \{3,\dots, n\}$, see (\ref{A.4}) and that $\mu_1 = 1$, see (\ref{1.4})) and define the affine isomorphism $h$ between $\IR^\Sigma$ and the affine hyperplane $\wt{H} = \{\wt{\sigma} = (\wt{s}_i)_{2 \le i \le n} \in \IR^{\wt{\Sigma}}$; $\Sigma_2^n \, \wt{s}_i = u-v\}$:
\begin{equation}\label{2.14}
\begin{array}{l}
h(\sigma) = \wt{\sigma} \; \mbox{with} \; \sigma = (s_i)_{3 \le i \le n}, \; \wt{\sigma} = (\wt{s}_i)_{2 \le  i \le n}, \; \mbox{and}
\\[1ex]
\wt{s}_i = (1 - \mu^{-1}_i) \; s_i, \; \mbox{for} \; 3 \le i \le n, \; \wt{s}_2 = u-v - \sum^n_3 (1- \mu_i^{-1}) \, s_i .
\end{array}
\end{equation}
Then, $h$ maps $\IV_{u,v} = \{(s_3,\dots, s_n) \in \IR^\Sigma_+$; $u - \mu_1 v \ge \Sigma^n_3 \,(1 - \frac{\mu_1}{\mu_i})\, s_i$ and $u - \mu_2 v \le \Sigma^n_3 \,(1 - \frac{\mu_2}{\mu_i}) \,s_i\}$, see (\ref{A.5}), one-to-one onto
\begin{equation}\label{2.15}
\begin{array}{l}
\wt{\IV}_{u,v} = \{ \wt{\sigma} = (\wt{s}_i)_{2 \le i \le n} \in \IR^{\wt{\Sigma}}_+; \; \Sigma_2^n \, \wt{s}_i = u-v \;\,\mbox{and} \; \,\Sigma_2^n \, \gamma_i \,\wt{s}_i \ge u - \mu_2 v\} \subseteq \wt{H},
\\[1ex]
\mbox{with} \; \gamma_i = \mbox{\f $\dis\frac{1 - \mu_2/\mu_i}{1 - \mu_1/\mu_i}$} = 1 -  \mbox{\f $\dis\frac{\mu_2-\mu_1}{ \mu_i - \mu_1}$} , \; \mbox{for $2 \le i \le n$ (and we recall that $\mu_1 = 1)$}.
\end{array}
\end{equation}
Note that the $\gamma_i$ are non-decreasing in $i$ and that $\gamma_2 = 0$. Further, $h$ maps the Lebesgue measure on $\IR^\Sigma$ into a constant multiple of the surface measure on the hyperplane $\wt{H}$, so that $h$ maps the barycenter of $\IV_{u,v}$ on the barycenter of $\wt{\IV}_{u,v}$, which we denote by $\wt{q}(u,v) = (\wt{q}_i (u,v))_{2 \le i \le n}$. We thus see by (\ref{B.16}) that
\begin{equation} \label{2.16}
 h\big(\wh{q}(u,v)\big) = \wt{q}(u,v), \; \mbox{with the notation $\wh{q} (u,v) = \big(\wh{q}_i (u,v)\big)_{3 \le i \le n}$, i.e.}
 \end{equation}
\begin{equation}\label{2.17}
\begin{split}
 \wt{q}_i (u,v)  &= (1 - \mu^{-1}_i) \, \wh{q}_i (u,v), \; \mbox{for $3 \le i \le n$},
\\
 \wt{q}_2(u,v) &= u-v - \Sigma^n_3 \,(1 - \mu_i^{-1})\, \wh{q}_i(u,v) \stackrel{(\ref{B.17})}{=} (1 - \mu^{-1}_2) \; \wh{q}_2 (u,v).  
 \end{split}
\end{equation}

\medskip\n
The claim (\ref{2.11}) thus amounts to showing that
\begin{equation}\label{2.18}
\mbox{$\wt{q}_i (u,v), 2 \le i \le n$, is non-decreasing in $i$}.
\end{equation}

\n
For this purpose we consider $i_0 \in \{2, \dots, n-1\}$ and the map $S$ swapping the $i_0$ coordinate with the $i_0 + 1$ coordinate on $\IR^{\wt{\Sigma}}$, namely, such that for $\wt{\sigma} = (\wt{s}_i)_{2 \le i \le n}$
\begin{equation}\label{2.19}
\begin{split}
S(\wt{\sigma})_i & = \wt{s}_i, \; \mbox{when $i \in \wt{\Sigma}\; \backslash\,\{i_0, i_0 + 1\}$},
\\
& = \wt{s}_{i_0 + 1}, \; \mbox{when $i = i_0$},
\\
& = \wt{s}_{i_0}, \; \mbox{when $i = i_0 + 1$}.
\end{split}
\end{equation}
We also introduce (see (\ref{2.15}) for notation)
\begin{equation}\label{2.20}
\wt{\IV}_{i_0, +} = \{\wt{\sigma} \in \wt{\IV}_{u,v}; \; \wt{s}_{i_0} < \wt{s}_{i_0 + 1}\}, \; \wt{\IV}_{i_0,-} = \{\wt{\sigma} \in \wt{\IV}_{u,v}; \; \wt{s}_{i_0 + 1} < \wt{s}_{i_0}\}.
\end{equation}
We then note that
\begin{equation}\label{2.21}
\begin{array}{l}
\mbox{$\wt{\IV}_{i_0,+}$, $\wt{\IV}_{i_0,-}$ are pairwise disjoint convex polytopes contained in $\wt{\IV}_{u,v}$}, 
\\
\mbox{and their union differs from $\wt{\IV}_{u,v}$ by a negligible set}.
\end{array}
\end{equation}
The key observation is now that
\begin{equation}\label{2.22}
\mbox{$S$ maps $\wt{\IV}_{i_0, -}$ into $\wt{\IV}_{i_0, +}$},
\end{equation}
because, see (\ref{2.15}) for notation,
\begin{equation}\label{2.23}
\gamma_{i_0} \, \wt{s}_{i_0} + \gamma_{i_0 + 1} \, \wt{s}_{i_0 + 1} < \gamma_{i_0} \, \wt{s}_{i_0 + 1} + \gamma_{i_0 + 1} \, \wt{s}_{i_0} \; \mbox{when} \; \wt{s}_{i_0 + 1} < \wt{s}_{i_0},
\end{equation}

\n
since  their difference equals $(\gamma_{i_0 + 1} - \gamma_{i_0}) (\wt{s}_{i_0} - \wt{s}_{i_0 + 1}) > 0$, so that $S(\wt{\IV}_{i_0, -}) \subseteq \wt{\IV}_{u,v}$, and (\ref{2.22}) readily follows.

\medskip
The map $S$ preserves the surface measure on $\wt{H}$, and we thus see that
\begin{equation}\label{2.24}
\mbox{the barycenter $\wt{q}(u,v)$ of $\wt{\IV}_{u,v}$ belongs to the closure of $\wt{\IV}_{i_0, +}$}.
\end{equation}

\n
This shows that $\wt{q}_{i_0}(u,v) \le \wt{q}_{i_0 + 1}(u,v)$ for all $i_0$ in $\{2,\dots,n\}$ and $(u,v)$ in $\stackrel{_\circ}{C}$. The claim (\ref{2.18}) follows, and as already observed this completes the proof of Theorem \ref{theo2.3}.
\end{proof}

\begin{remark}\label{rem2.4} \rm In the case $0 \le v \le u \le \mu_2 v$, one has the explicit formulas, see (\ref{2.10}), $\wh{q}_i (u,v) = [(n-1) (1 - \mu_i^{-1})]^{-1} (u-v)$, for $2 \le i \le n$. So the upper bound (\ref{2.12}) has a good quality when $i$ remains small compared to $n$. Also, when $u = \mu_n v$, we know by (\ref{2.6}) that $\wh{q}_n (u,v) = u = (1-\mu_n^{-1})^{-1} (u-v)$, so that the upper bound (\ref{2.12}) is sharp in this case. \hfill $\square$

\end{remark}

\section{The effective diffusion}
\setcounter{equation}{0}

In this section, we construct by means of a martingale problem, see Theorem \ref{theo3.2}, the diffusion in the interior of the two-dimensional cone $C$ defined in (\ref{1.11}).  Its generator is formally obtained by replacing the terms $x^2_\ell$ in (\ref{1.10}) by the functions $q_\ell$ of Section 2. In the course of the proof we establish an important Lyapunov-Foster condition, see (\ref{3.21}) of Proposition \ref{prop3.3}.

\medskip
With this in mind we introduce the diffusion matrix $\wt{a}$ and the drift $\wt{b}$ via
\begin{equation}\label{3.1}
\wt{a} = 4a \left(\begin{array}{ll}
\Sigma_\ell \, \lambda_\ell (1 + \delta_\ell)\,q_\ell & \Sigma_\ell (1 + \delta_\ell)\,q_\ell
\\[1ex]
\Sigma_\ell \, (1 + \delta_\ell)\,q_\ell & \Sigma_\ell \; \frac{1}{\lambda_\ell} \; (1 + \delta_\ell)\,q_\ell
\end{array}\right), \; \wt{b}  =  \left(\begin{array}{ll}
a\, \Sigma_\ell \, \lambda_\ell (1 + \delta_\ell) - 2 \Sigma_\ell \,\lambda_\ell \,q_\ell\\[1ex]
a\, \Sigma_\ell \, (1 + \delta_\ell) - 2 \Sigma_\ell \,q_\ell
\end{array}\right)\,.
\end{equation}
The next lemma provides meaning to the above terminology.

\begin{lemma}\label{lem3.1}
\begin{align}
&\mbox{$\wt{a}$ and $\wt{b}$ are Lipschitz continuous functions, homogeneous of degree $1$ on $C$.} \label{3.2}
\\[1ex]
& \mbox{$\wt{a}$ is positive definite on $\stackrel{_\circ}{C}$.} \label{3.3}
\end{align}
\end{lemma}

\begin{proof}
The claim (\ref{3.2}) is a direct consequence of (\ref{2.8}) and (\ref{2.1}). To prove (\ref{3.3}), we compute the trace and the determinant of $\wt{a}$. We have
\begin{equation}\label{3.4}
\mbox{$tr \, \wt{a} = 4a \, \Sigma_\ell \, \Big(\lambda_\ell + \mbox{\f $\dis\frac{1}{\lambda_\ell}$}\Big) (1 + \delta_\ell) \, q_\ell > 0$ on $\stackrel{_\circ}{C}$, by (\ref{2.1}) and (\ref{1.7})}.
\end{equation}
Moreover,
\begin{equation}\label{3.5}
\begin{array}{lcl}
&\hspace{-5ex} {\rm det} \,\wt{a}  = &\hspace{-7ex} 16 a^2  \dis\sum_{1 \le  \ell, m \le N} \, \Big(\mbox{\f $\dis\frac{\lambda_\ell}{\lambda_m}$} - 1\Big) (1 + \delta_\ell) \,q_\ell \,(1 + \delta_m) \, q_m
\\[3ex]
& \stackrel{\mbox{\small symmetrizing}}{=} & 8a^2 \dis\sum_{1 \le  \ell, m \le N} \,\Big(\mbox{\f $\dis\frac{\lambda_\ell}{\lambda_m}$} + \mbox{\f $\dis\frac{\lambda_m}{\lambda_\ell}$} - 2 \Big)  (1 + \delta_\ell) \, q_\ell\,(1 + \delta_m) \, q_m
\\[3ex]
& \stackrel{\mbox{\small (\ref{1.2}),(\ref{2.2})}}{\ge} &\hspace{-3ex}  32 a^2 \,\inf\limits_{1 \le i \not= j \le n} \,\Big(\mbox{\f $\dis\frac{\mu_i}{\mu_j}$} + \mbox{\f $\dis\frac{\mu_j}{\mu_i}$} - 2\Big) \, \dis\sum_{1 \le i \not= j \le n} \,(1 + \delta_{2i}) \,q_{2i} \,(1 + \delta_{2j}) \, q_{2j}
\\[3ex]
& > & \hspace{-6ex} \mbox{$0$ on $\stackrel{_\circ}{C}$ by (\ref{2.1}) and (\ref{1.2})}.
\end{array}
\end{equation}
Hence, on $\stackrel{_\circ}{C}$ the eigenvalues of $\wt{a}$ are positive and (\ref{3.3}) follows. This proves Lemma \ref{lem3.1}.
\end{proof}

We can thus introduce the diffusion operator on $\stackrel{_\circ}{C}$ attached to $\wt{a}$ and $\wt{b}$; it acts on smooth functions $\psi$ compactly supported in $\stackrel{_\circ}{C}$ via
\begin{equation}\label{3.6}
\begin{split}
\wt{A} \psi = & \; 2a \Big\{\Sigma_\ell \, \lambda_\ell \,(1 + \delta_\ell) \,q_\ell \, \partial^2_u \, \psi + 2 \Sigma_\ell (1 + \delta_\ell) \,q_\ell \,\partial^2_{u,v} \psi + \Sigma_\ell \, \mbox{\f $\dis\frac{1}{\lambda_\ell}$} \; (1 + \delta_\ell) \,q_\ell \ \partial^2_v \psi\Big\}
\\[1ex]
& +  \big(a \Sigma_\ell \, \lambda_\ell \,(1 + \delta_\ell)  - 2 \Sigma_\ell \,\lambda_\ell \, q_\ell\big) \, \partial_u \psi + \big( a \Sigma_\ell (1 + \delta_\ell) - 2 \Sigma_\ell \,q_\ell \big) \,\partial_v \psi
\end{split}
\end{equation}

\n
(by Lemma \ref{lem3.1} this is an elliptic second order operator with Lipschitz coefficients in $\stackrel{_\circ}{C}$, which are homogeneous of degree $1$). As an aside, there is no gain of generality in allowing that $\delta_{2i}$ differs from $\delta_{2i-1}$ in (\ref{1.7}). Indeed, due to the fact that $q_{2i} = q_{2i-1}$ and $\lambda_{2i} = \lambda_{2i-1}$, see (\ref{2.2}), (\ref{1.2}), $\wt{A}$ remains identical to the operator that one obtains with $\delta_{2i}$ and $\delta_{2i-1}$ replaced by the half of their sum.

\medskip
We then consider the space $C(\IR_+, \stackrel{_\circ}{C})$ of continuous $\stackrel{_\circ}{C}$-valued functions on $\IR_+$, endowed with its canonical $\sigma$-algebra $\cF$, its canonical filtration $(\cF_t)_{t \ge 0}$, and its canonical process $(W_t)_{t \ge 0}$. Our main object is the well-posedness of the martingale problem for $\wt{A}$ in $\stackrel{_\circ}{C}$.
\begin{theorem}\label{theo3.2}
Given $w \in\, \stackrel{_\circ}{C}$, there is a unique probability $\wt{P}_w$ on $C(\IR_+, \stackrel{_\circ}{C}$) such that
\begin{equation}\label{3.7}
\begin{array}{l}
\mbox{$\wt{P}_w (W_0 = w) = 1$, and for any smooth compactly supported function $\psi$ on $\stackrel{_\circ}{C}$},
\\[1ex]
\mbox{$\psi(W_t) - \dis\int^t_0 \wt{A} \, \psi (W_s)\, ds, t \ge 0$, is an $(\cF_t)_{t \ge 0}$-martingale}.
\end{array}
\end{equation}
\end{theorem}

\begin{proof}
We approximate $\stackrel{_\circ}{C}$ from inside, and for
\begin{equation}\label{3.8}
0 < \eta < (\lambda_N - 1) \, / \, 2,
\end{equation}

\vspace{-1ex}\n
define the open subset of $\stackrel{_\circ}{C}$
\begin{equation}\label{3.9}
C_\eta = \{w = (u,v) \in \, \stackrel{_\circ}{C}; 1 + \eta < u / v < \lambda_N - \eta \;\mbox{and} \; v > \eta\}
\end{equation}
so that with Lemma \ref{lem3.1}

\vspace{-4ex}
\begin{equation}\label{3.10}
\mbox{$\wt{a}$ is uniformly elliptic on $C_\eta$}.
\end{equation}
We then consider $w \in\, \stackrel{_\circ}{C}$ as in (\ref{3.7}), and assume that $\eta > 0$ in (\ref{3.8}) is small enough so that $w \in C_\eta$. By Theorem 5.2.2, p.~131 of \cite{StroVara79}, $\wt{a}^{1/2}(\cdot)$ is a Lipschitz function on $\ov{C}\!_\eta$. This is also the case for $\wt{b}(\cdot)$. This fact ensures that the law we are looking for, is uniquely determined up to the exit time of $C_\eta$. Specifically, we can consider arbitrary Lipschitz extensions to $\IR^2$ of the restriction to $C_\eta$ of $\wt{a}^{1/2}$ and $\wt{b}$, denote by $A_\eta$ the corresponding diffusion operator, so that $A_\eta \psi$ coincides with $\wt{A} \psi$ on $\ov{C}\!_\eta$, when $\psi$ is supported in $\stackrel{_\circ}{C}$. The martingale problem attached to $A_\eta$ is well-posed, see for instance \cite{KaraShre88}, p.~319, and letting $T_\eta$ stand for the exit time of the canonical process (still denoted by $(W_t)_{t \ge 0})$ and $P^\eta_w$ for the corresponding law with starting point $w$ the law of $(W_{t \wedge T_\eta})_{t \ge 0}$ under $P^\eta_w$ does not depend on the choice of the extensions. By Corollary 10.1.2, p.~250 of \cite{StroVara79}, Theorem \ref{theo3.2} will be proved once we show that
\begin{equation}\label{3.11}
\mbox{for any fixed $T > 0$, $\lim\limits_{\eta \r 0} \, P^\eta_w [T_\eta \le T] = 0$}.
\end{equation}
To establish (\ref{3.11}), we will construct a smooth non-negative function $\Phi$ on $\stackrel{_\circ}{C}$, with sub-level sets $\{\Phi \le t\}$, which are for every $t \ge 0$, (possibly empty) compact subsets of $\stackrel{_\circ}{C}$, a non-empty compact subset $\Gamma$ of $\stackrel{_\circ}{C}$, and a positive real number $C_\Phi$, so that
\begin{align}
& \mbox{$\wt{A} \Phi \le -1$ on $\stackrel{_\circ}{C} \, \backslash \, \Gamma$, and} \label{3.12}
\\[1ex]
& \mbox{$\wt{A} \Phi \le C_\Phi$ on $\Gamma$} \label{3.13}
\end{align}
(where $\wt{A}$ in (\ref{3.6}) has been tacitly extended to smooth functions on $\stackrel{_\circ}{C}$).

\medskip
This is a so-called Lyapunov-Foster condition. By Theorem 2.1 on p.~524 of \cite{MeynTwee93} it is (much) more than what is needed to ensure that (\ref{3.11}) holds and complete the proof of Theorem \ref{theo3.2}. But this Lyapunov-Foster condition will be helpful in the next section to establish that the diffusion on $\stackrel{_\circ}{C}$ constructed in this section has a unique stationary distribution.

\begin{proposition}\label{prop3.3}
Define the function $\Phi$ on $\stackrel{_\circ}{C}$ via
\begin{equation}\label{3.14}
\Phi = \Phi_F + \Phi_G + \Phi_V,
\end{equation}
where for $w = (u,v) \in\, \stackrel{_\circ}{C}$
\begin{equation}\label{3.15}
\Phi_F(w) = (u-v)^{-\alpha_0}, \Phi_G(w) = (\lambda_N v-u)^{-\beta_0}, \Phi_V(w) = e^{\gamma_0 v},
\end{equation}
and $\alpha_0, \beta_0, \gamma_0 > 0$ are chosen so that
\begin{align}
& 2(2 \alpha_0 + 1) \, \max\limits_\ell (\lambda_\ell - 1) (1 + \delta_\ell) < \Sigma_\ell (\lambda_\ell - 1) (1 + \delta_\ell), \label{3.16}
\\[1ex]
& 2(2 \beta_0 + 1) \, \max\limits_\ell (\lambda_N - \lambda_\ell) (1 + \delta_\ell) < \Sigma_\ell (\lambda_N  - \lambda_\ell) (1 + \delta_\ell), \label{3.17}
\end{align}
(such choices are possible because $2 \max_\ell \,(\lambda_\ell - 1)(1 + \delta_\ell) < \Sigma_\ell (\lambda_\ell -1)(1 + \delta_\ell)$ and $2 \max_\ell \,(\lambda_N - \lambda_\ell) (1 + \delta_\ell) < \Sigma_\ell (\lambda_N - \lambda_\ell) (1 + \delta_\ell)$, by {\rm (\ref{1.1}), (\ref{1.2})}, and {\rm (\ref{1.7})}), 
\begin{equation}\label{3.18}
\gamma_0 = \mbox{\f $\dis\frac{1}{2a}$}\,.
\end{equation}

\n
Then, one has $c_F, c_G, c_V > 0$, so that
\begin{equation}\label{3.19}
\mbox{$\wt{A} \Phi_F \le c_F, \wt{A} \Phi_G \le c_G$, and $\wt{A} \Phi_V \le c_V$ on $\stackrel{_\circ}{C}$},
\end{equation}
and one can choose $f_0, g_0, v_0$ positive such that setting
\begin{align}
&\Gamma = \{w = (u,v) \in \, \stackrel{_\circ}{C}; u-v \ge f_0, \lambda_N v - u \ge g_0, v \le v_0\}, \; \mbox{and} \;\, c_\Phi = c_F + c_G + c_V, \label{3.20}
\\[2ex]
&\mbox{$\Phi, \Gamma, c_\Phi$ satisfy the Lyapunov-Foster condition corresponding to {\rm (\ref{3.12}), (\ref{3.13})}} \label{3.21}
\\
&\mbox{and the assumptions above {\rm (\ref{3.12}), (\ref{3.13})}}. \nonumber
\end{align}
\end{proposition}

\begin{proof}
We set $f(w) = u-v \, (> 0)$ for $w = (u,v) \in \,\stackrel{_\circ}{C}$, and find by (\ref{3.6}) that for $\alpha > 0$,
\begin{equation}\label{3.22}
\begin{split}
\wt{A} f^{-\alpha} = &  - \mbox{\f $\dis\frac{\alpha}{f^{\alpha + 1}}$} \,(a \Sigma_\ell \,(\lambda_\ell - 1) (1 + \delta_\ell) - 2 \Sigma_\ell (\lambda_\ell - 1) \,q_\ell)
\\[1ex]
& +  \mbox{\f $\dis\frac{\alpha (\alpha + 1)}{f^{\alpha + 2}}$}  \; 2a \, \Sigma_\ell \,\Big(\lambda_\ell + \mbox{\f $\dis\frac{1}{\lambda_\ell}$} - 2\Big) \,(1+ \delta_\ell) \, q_\ell.
\end{split}
\end{equation}
Noting that $\lambda_\ell + \frac{1}{\lambda_\ell} - 2 = (\lambda_\ell - 1)(1 - \frac{1}{\lambda_\ell})$ and $f = \Sigma_\ell (1 - \frac{1}{\lambda_\ell}) \,q_\ell$ by (\ref{2.3}), we find that
\begin{equation}\label{3.23}
\begin{split}
\wt{A} \, f^{-\alpha}  = &- \mbox{\f $\dis\frac{\alpha}{f^{\alpha + 1}}$} \,\Big\{a \Sigma_\ell \,(\lambda_\ell - 1)(1 + \delta_\ell) 
\\
&- 2a (\alpha + 1) \, \Sigma_\ell \,(\lambda_\ell - 1) (1 + \delta_\ell) \Big(1 - \mbox{\f $\dis\frac{1}{\lambda_\ell}$}\Big) \, q_\ell \, / \, \Sigma_{\ell^\prime}\, \Big(1 - \mbox{\f $\dis\frac{1}{\lambda_{\ell^\prime}}$}\Big) \, q_{\ell^\prime} - 2 \Sigma_\ell \,(\lambda_\ell - 1) \,q_\ell\Big\}.
\end{split}
\end{equation}
If we now choose $\alpha = \alpha_0$, with $\alpha_0$ as in (\ref{3.16}), we note that
\begin{equation*}
\begin{array}{l}
2 (\alpha_0 + 1) \, \Sigma_\ell \,(\lambda_\ell -1) (1 + \delta_\ell) \, \Big(1 - \mbox{\f $\dis\frac{1}{\lambda_\ell}$}\Big) \, q_\ell  \, / \, \Sigma_{\ell^\prime} \Big(1 - \mbox{\f $\dis\frac{1}{\lambda_{\ell^\prime}}$}\Big) \, q_{\ell^\prime} \le 2 (\alpha_0 + 1) \, \max\limits_\ell \,\{(\lambda_\ell - 1)(1 + \delta_\ell)\}
\\
\le \mbox{\f $\dis\frac{\alpha_0 +1}{2 \alpha_0 + 1}$} \; \Sigma_\ell\, (\lambda_\ell - 1) (1 + \delta_\ell),
\end{array}
\end{equation*}
so that coming back to (\ref{3.23}), and using $\Sigma_\ell \,(\lambda_\ell - 1) \,q_\ell = \Sigma_\ell \,\lambda_\ell \, (1 - \frac{1}{\lambda_\ell}) \, q_\ell \le \lambda_N f$, we find that on $\stackrel{_\circ}{C}$
\begin{equation}\label{3.24}
\wt{A} f^{-\alpha_0} \le - \mbox{\f $\dis\frac{\alpha_0}{f^{\alpha_0 + 1}}$} \; \Big\{ \mbox{\f $\dis\frac{a \alpha_0}{2 \alpha_0 + 1}$} \; \Sigma_\ell (\lambda_\ell - 1)(1 + \delta_\ell) - 2 \lambda_N f\Big\}.
\end{equation}
Further, note that the right member of (\ref{3.24}) is uniformly bounded from above, so that the first inequality of (\ref{3.19}) holds (see (\ref{3.15}) for notation).

\medskip
We then set $g(w) = \lambda_N v - u \,(> 0)$ for $w \in \, \stackrel{_\circ}{C}$, and find by (\ref{3.6}) that for $\beta > 0$,
\begin{equation}\label{3.25}
\begin{split}
\wt{A} g^{-\beta} = & - \mbox{\f $\dis\frac{\beta}{g^{\beta + 1}}$} \; \big(a \Sigma_\ell (\lambda_N - \lambda_\ell)(1 + \delta_\ell) - 2 \Sigma_\ell (\lambda_N - \lambda_\ell) \,q_\ell \big)
\\[1ex]
& +  \mbox{\f $\dis\frac{\beta (\beta + 1)}{g^{\beta + 2}}$} \; 2 a \, \lambda_N \, \Sigma_\ell \, \Big(\mbox{\f $\dis\frac{\lambda_N}{\lambda_\ell}$} +  \mbox{\f $\dis\frac{\lambda_\ell}{\lambda_N}$} -2 \Big) (1 + \delta_\ell) \, q_\ell .
\end{split}
\end{equation}
With a similar observation as above (\ref{3.23}) we obtain
\begin{equation}\label{3.26}
\begin{split}
\wt{A} g^{-\beta} = & - \mbox{\f $\dis\frac{\beta}{g^{\beta + 1}}$} \; \big(a \Sigma_\ell (\lambda_N - \lambda_\ell)(1 + \delta_\ell)
\\
& - 2 a (\beta + 1)\, \Sigma_\ell (\lambda_N - \lambda_\ell) (1 + \delta_\ell) \,\Big(\mbox{\f $\dis\frac{\lambda_N}{\lambda_\ell}$} - 1\Big)\,q_\ell  \, / \, \Sigma_{\ell^\prime} \,  \Big(\mbox{ \f $\dis\frac{\lambda_N}{\lambda_{\ell^\prime}}$} - 1\Big)\, q_{\ell^\prime} 
\\
& - 2 \Sigma_\ell \,(\lambda_N - \lambda_\ell) \, q_\ell\Big\}.
\end{split}
\end{equation}
If we now choose $\beta = \beta_0$, with $\beta_0$ as in (\ref{3.17}), we find that 
\begin{equation*}
\begin{array}{l}
2 (\beta_0 + 1) \, \Sigma_\ell (\lambda_N - \lambda_\ell) (1 + \delta_\ell)\, \Big( \mbox{\f $\dis\frac{\lambda_N}{\lambda_\ell}$}  - 1\Big) \, q_\ell \, / \, \Sigma_{\ell^\prime}\,   \Big( \mbox{\f $\dis\frac{\lambda_N}{\lambda_{\ell^\prime}}$}  - 1\Big)\, q_{\ell^\prime} \le 2 (\beta_0 + 1) \, \max \{(\lambda_N - \lambda_\ell) (1 + \delta_\ell)\}
\\
\le \mbox{\f $\dis\frac{\beta_0 + 1}{2 \beta_0 + 1}$} \; \Sigma_\ell (\lambda_N - \lambda_\ell) (1 + \delta_\ell).
\end{array}
\end{equation*}
Thus, coming back to (\ref{3.26}), using that $\Sigma_\ell \,(\lambda_N - \lambda_\ell) \, q_\ell = \Sigma_\ell \, \lambda_\ell \, (\frac{\lambda_N}{\lambda_\ell}  - 1)\, q_\ell \le \lambda_N g$, we obtain
\begin{equation}\label{3.27}
\wt{A} g^{-\beta_0} \le - \mbox{\f $\dis\frac{\beta_0}{g^{\beta + 1}}$} \;\Big\{ \mbox{\f $\dis\frac{a \beta_0}{2\beta _0 + 1}$} \; \Sigma_\ell (\lambda_N - \lambda_\ell) (1 + \delta_\ell) - 2 \lambda_N g\Big\}.
\end{equation}
We also note that the right member of (\ref{3.27}) is uniformly bounded from above, so that the second inequality of (\ref{3.19}) holds (see (\ref{3.15}) for notation).

\medskip
Finally, we have for $\gamma \in (0,(2a)^{-1}]$ and $h_\gamma(w)  = e^{\gamma v}$ with $w = (u,v) \in \, \stackrel{_\circ}{C}$,
\begin{equation}\label{3.28}
\begin{split}
\wt{A} h_\gamma(w)  = & \Big\{ 2 a \, \gamma^2 \, \Sigma_\ell \; \mbox{\f $\dis\frac{1}{\lambda_\ell}$} (1 + \delta_\ell)\, q_\ell + \gamma\,\big(a \Sigma_\ell ( 1+ \delta_\ell) - 2u\big)\Big\} \, e^{\gamma v} 
\\
&\hspace{-4.9ex} \stackrel{(\ref{1.7}),(\ref{2.3})}{\le} (2 a\, \gamma v - 2 u + a N) \, \gamma e^{\gamma v} \le ( - u + a N) \, \gamma e^{\gamma v} (\le a N \, \gamma e^{\gamma a N}).
\end{split}
\end{equation}
Thus, we choose
\begin{align}
&\gamma_0 = (2a)^{-1} \; \mbox{and} \; c_V = \mbox{\f $\dis\frac{N}{2}$} \;e^{\frac{N}{2}} , \mbox{so that} \label{3.29}
\\[1ex]
& \wt{A} h_{\gamma_0} \le (-u + a N) \, / \, 2 a \;\, e^{v / 2a} \le c_V, \label{3.30}
\end{align}
and in particular the last inequality of (\ref{3.19}) holds (see (\ref{3.15})).

\medskip
Then, combining (\ref{3.19}) with (\ref{3.24}), (\ref{3.27}), (\ref{3.30}), we see that we can choose $f_0 > 0$, $g_0 > 0$ (small enough) and $v_0 > 0$ (large enough) so that on $\{f \le f_0\} \cup \{g \le g_0\} \cup \{w \in \,\stackrel{_\circ}{C}$; $v \ge v_0\}$, $\wt{A} ( \Phi_F + \Phi_G + \Phi_V) \le -1$, i.e.~$\wt{A} \Phi \le -1$. Since $\Phi$ is smooth, non-negative, and its sub-level sets are compact subsets of $\stackrel{_\circ}{C}$, this completes the proof of (\ref{3.21}) and hence of Proposition \ref{prop3.3}.
\end{proof}

As noted below (\ref{3.13}), Proposition \ref{prop3.3} completes the proof of (\ref{3.11}), and hence of Theorem \ref{theo3.2}.
\end{proof}

We finally record a consequence of the proof of Proposition \ref{prop3.3}, which will be useful in the proof of (\ref{4.2}) in the next section.

\begin{corollary}\label{cor3.4}
With $\alpha_0, \beta_0, \gamma_0 > 0$ and $\Phi$ as in Proposition \ref{prop3.3}, 
\begin{equation}\label{3.31}
\begin{array}{l}
\mbox{the function $|\wt{A} \Phi(w)| \, / \, \{(u-v)^{-\alpha_0 - 1} + (\lambda_N \, v-u)^{- \beta_0 - 1} + e^{\gamma_0 v}\}$ is}
\\
\mbox{bounded away from $0$ outside some compact subset of $\stackrel{_\circ}{C}$}.
\end{array}
\end{equation}
\end{corollary}

\begin{proof}
Choosing $0 < f_1 \le f_0$, $0 < g_1 \le g_0$, small enough and $v_1 \ge v_0$ large enough, we can ensure that the sum of the right members of (\ref{3.24}), (\ref{3.27}), and of the first inequality of (\ref{3.30}) is smaller than $- \rho_F (u-v)^{-(\alpha_0 + 1)}$ on $\{f \le f_1\}$, $- \rho_G(\lambda_N \, v-u)^{-(\beta_0 + 1)}$ on $\{ g \le g_1\}$ and $-\rho_V \, e^{\gamma_0 v}$ on $\{v \ge v_1\}$, with $\rho_F, \rho_G, \rho_V$ positive. It then follows that on $\{f \le f_1\} \cup \{ g \le g_1\} \cup \{v \ge v_1\}$ the ratio in (\ref{3.31}) remains bounded away from zero. In addition, $\{f \ge f_1\} \cap \{ g \ge g_1\} \cap \{ v \le v_1\}$ is a compact subset of  $\stackrel{_\circ}{C}$. This proves Corollary \ref{cor3.4}.
\end{proof}

\begin{remark}\label{rem3.5} \rm Given $w \in \Delta_2$, see (\ref{2.1}), under $\wt{P}_w$, the process $F_t = U_t - V_t$, $t \ge 0$, is by Theorem \ref{theo3.2} and stochastic calculus a continuous $(\cF_t)_{t \ge 0}$-semi-martingale with drift term equal to
\begin{equation}\label{3.32}
\dis\int^t_0 a \Sigma_\ell (\lambda_\ell -1)(1 + \delta_\ell) - 2 \Sigma_\ell \,(\lambda_\ell - 1) \, q_\ell (W_s) \, ds, t \ge 0,
\end{equation}
and bracket process equal to
\begin{equation}\label{3.33}
\begin{split}
\langle F\rangle_t =  &\,4 a \dis\int^t_0 \Sigma_\ell \, \Big(\lambda_\ell + \mbox{\f $\dis\frac{1}{\lambda_\ell}$} - 2 \Big) ( 1 + \delta_\ell) \, q_\ell (W_s) \, ds 
\\
= & \,4 a \dis\int^t_0  \Sigma_\ell \, (\lambda_\ell - 1) \Big(1 - \mbox{\f $\dis\frac{1}{\lambda_\ell}$}\Big) ( 1 + \delta_\ell) \, q_\ell \,(W_s) \, ds, t \ge 0.
\end{split}
\end{equation}

\n
If $\tau = \inf\{ t \ge 0; W_s \notin \Delta_2\}$ denote the exit time of $W_t, t \ge 0$, from the sector $\Delta_2$, see (\ref{2.1}), and we use the explicit formulas (\ref{2.10}) for the functions $q_\ell$ in $\ov{\Delta}_2$, we find that under $\wt{P}_w$, the stopped process $F_{t \wedge \tau}, t \ge 0$, is a continuous $(\cF_t)_{t \ge 0}$-semi-martingale with drift term equal to
\begin{equation}\label{3.34}
\dis\int_0^{t \wedge \tau} a \Sigma_\ell \,(\lambda_\ell - 1)(1 + \delta_\ell) - \mbox{\f $\dis\frac{2}{N-2}$} \;(\Sigma_{\ell \ge 3} \,\lambda_\ell) \, F_s \, ds, t \ge 0,
\end{equation}
and bracket process equal to
\begin{equation}\label{3.35}
\langle F \rangle_{t \wedge \tau} = \mbox{\f $\dis\frac{4a}{N-2}$} \; \Sigma_\ell \, (\lambda_\ell - 1)(1 + \delta_\ell) \; \dis\int_0^{t \wedge \tau} F_s \, ds, t \ge 0.
\end{equation}

\n
It readily follows from Ito's formula that for any $C^2$-function $\rho$ on $\IR$, 
\begin{equation}\label{3.36}
\rho (F_{t \wedge \tau}) - \dis\int^{t \wedge \tau}_0 \cH \rho(F_s) \, ds, t \ge 0,
\end{equation}
is a continuous $(\cF_t)_{t \ge 0}$-local martingale under $\wt{P}_w$, where we have set 
\begin{equation}\label{3.37}
\begin{array}{l}
\cH \, \rho(x) = A x \rho^{\prime\prime}(x) + (C x + D) \, \rho^\prime(x), \;\mbox{for $x \in \IR$, with}
\\[0.5ex]
A = \mbox{\f $\dis\frac{2a}{N-2}$} \; \Sigma_\ell\, (\lambda_\ell - 1) (1 + \delta_\ell), \; C = - \mbox{\f $\dis\frac{2}{N-2}$} \; \Sigma_{\ell \ge 3} \, \lambda_\ell, \; D = a \Sigma_\ell \, (\lambda_\ell -1 ) (1 + \delta_\ell).
\end{array}
\end{equation}

\n
This can be viewed as expressing that under $\wt{P}_w$, up to the exit time of $(W_t)_{t \ge 0}$ from $\Delta_2$, the positive process $F_t = U_t - V_t, t \ge 0$, behaves as a one-dimensional diffusion with generator $\cH$. Note that $D > A$, and the one-dimensional diffusion on $(0,\infty)$ attached to $\cH$ does not reach $0$, see \cite{IkedWata89}, pp.~236-237. \hfill $\square$
\end{remark}

\section{The stationary distribution}
\setcounter{equation}{0}

As seen in Theorem \ref{3.2} of the previous section, the martingale problem attached to $\wt{A}$ on $C(\IR_+, \stackrel{_\circ}{C})$ is well-posed for all initial conditions $w \in \,\stackrel{_\circ}{C}$. The corresponding family of solutions $\wt{P}_w$, $w \in \, \stackrel{_\circ}{C}$, thus provides a strong Markov process on the state space $\stackrel{_\circ}{C}$, see for instance Chapter 10 \S 1 of \cite{StroVara79}. In this section we will show that this Markov process has a unique stationary distribution and we will collect some of its properties, see Theorem \ref{4.1} and Propositions \ref{prop4.2} to \ref{4.4}.

\medskip
The first main result of this section is
\begin{theorem}\label{theo4.1}
\begin{equation}\label{4.1}
\begin{array}{l}
\mbox{There is a unique stationary probability $\wt{\pi}$ for the diffusion on $\stackrel{_\circ}{C}$ given by}
\\
\mbox{the collection $\wt{P}_w$, $w \in\, \stackrel{_\circ}{C}$ in Theorem \ref{theo3.2}}.
\end{array}
\end{equation}
Moreover, with $\alpha_0, \beta_0, \gamma_0$ as in (\ref{3.16}), (\ref{3.17}), (\ref{3.18}), one has
\begin{equation}\label{4.2}
\dis\int_{\stackrel{_\circ}{C}} \,(u-v)^{-(\alpha_0 +1)} + (\lambda_N v - u)^{-(\beta_0 +1)} + e^{\gamma_0 v} \, d \,\wt{\pi}(w) < \infty \;\;\mbox{(with $w = (u,v)$)}.
\end{equation}
\end{theorem}

\begin{proof}
Specifically, we will show that the diffusion on $\stackrel{_\circ}{C}$ constructed in the last section is ``positive Harris recurrent'', see Section 4 of \cite{MeynTwee93}, see also \cite{Lind92}, p.~91-98, for further background.

\medskip
With the Lyapunov-Foster condition (\ref{3.21}) in Proposition \ref{prop3.3}, an important part of the task has been completed. It will then follow from the application of Theorem 4.2 on p.~529 of \cite{MeynTwee93} that (\ref{4.1}) holds, and with the help of Corollary \ref{cor3.4} that (\ref{4.2}) holds, once we show that
\begin{equation}\label{4.3}
\mbox{for some $\eta > 0$, $\wt{P}_w[W_1 \in dx] \ge \eta \, 1_\Gamma \, dx$, for all $w \in \Gamma$}
\end{equation}
(we recall that $(W_t)_{t \ge 0}$ stands for the canonical process on $C(\IR_+, {\stackrel{_\circ}{C}})$, and $\Gamma$ has been defined in (\ref{3.20})).

\medskip
To this end we consider some smooth relatively compact domain $O$ such that $\Gamma \subseteq O \subseteq \ov{O} \subseteq \, \stackrel{_\circ}{C}$, and denote by $T_O = \inf\{s \ge 0; W_s \not= O\}$ the exit time from $O$. As we now explain, for a suitable (small) $\gamma > 0$, one has
\begin{equation}\label{4.4}
\mbox{$\wt{P}_w [ W_{1 \wedge T_O} \in dx] \ge \gamma \, 1_\Gamma \, dx$, for all $w \in \Gamma$},
\end{equation}

\n
which readily implies (\ref{4.3}) (since the probability in the left member is the sum of $\wt{P}_w[W_1 \in dx, T_O > 1]$ and the measure $\wt{P}_w[W_{T_O} \in dx, T_O \le 1]$ which is carried by the frontier of $O$ and hence orthogonal to $1_\Gamma\, dx$). To prove (\ref{4.4}), we use an extension outside $\ov{O}$ of the diffusion matrix $\wt{a}$ and the drift $\wt{b}$, see (\ref{3.1}), to uniformly Lipschitz functions, with uniform ellipticity for the extension of $\wt{a}$. We use Duhamel's formula to derive a small time lower bound on the Dirichlet heat kernel in $O$ for nearby points in a compact neighborhood of $\Gamma$ contained in $O$, with the help of the full space heat kernel estimates (see Theorem 1 on p.~67 and (4.75) on p.~82 of \cite{IlinKalaOlei62}, and then a chaining argument to deduce (\ref{4.4})). This completes the proof of (\ref{4.3}) and hence of Theorem \ref{theo4.1}.
\end{proof}

\medskip
We then introduce the stationary law of the effective diffusion: 
\begin{equation}\label{4.5}
\wt{P} = \dis\int \wt{\pi} (dw) \, \wt{P}_w \quad \mbox{(a probability on $C(\IR_+, \stackrel{_\circ}{C})$)}.
\end{equation}
Then, one has the following convenient martingale characterization of $\wt{P}$ via: 
\begin{proposition}\label{prop4.2}
$\wt{P}$ is the only probability $Q$ on $C(\IR_+, \stackrel{_\circ}{C})$ such that $Q$ is stationary and for any smooth compactly supported function $\psi$ on $C(\IR_+, \stackrel{_\circ}{C})$, under $Q$
\begin{equation*}
\mbox{$\psi (W_t) - \dis\int^t_0 \,\wt{A}\psi (W_s) \, ds, t \ge 0$, is an $(\cF_t)_{t \ge 0}$-martingale}
\end{equation*}
(recall $(\cF_t)_{t \ge 0}$ is the canonical filtration on $C(\IR_+, \stackrel{_\circ}{C})$, see above Theorem \ref{theo3.2}).
\end{proposition}

\begin{proof}
That $\wt{P}$ satisfies these conditions is immediate. Conversely, given a stationary law $Q$ as above, conditioning on $W_0 = w \in \, \stackrel{_\circ}{C}$, the resulting law is a solution to the martingale problem (\ref{3.7}), and hence coincides with $\wt{P}_w$, for a.e.~$w$ relative to the law of $W_0$ under $Q$. The stationarity of the law of $Q$ and (\ref{4.1}) implies that the law of $W_0$ under $Q$ is $\wt{\pi}$ and hence $Q = \wt{P}$. This proves Proposition \ref{prop4.2}.
\end{proof} 

\medskip
The above martingale characterization of $\wt{P}$ will be handy in the companion article \cite{SzniWidm26b} to study the inviscid limit of a certain enstrophy-energy process.

\medskip
Below are some further facts concerning $\wt{\pi}$.

\begin{proposition}\label{prop4.3}
\begin{equation}\label{4.6}
\mbox{$\wt{\pi}$ is absolutely continuous with respect to the Lebesgue measure on $\stackrel{_\circ}{C}$}.
\end{equation}
\end{proposition}

\begin{proof}
One can let the open set $O$ above (\ref{4.4}) grow along a sequence $O_n$, $n \ge 0$, of relatively open compact subsets of $\stackrel{_\circ}{C}$ (with $\ov{O}_n \subseteq O_{n+1}$ for each $n$). With the help of the heat kernel bounds in Theorem 1 on p.~67 of  \cite{IlinKalaOlei62}, one finds that for $m > n$, the law of $W_{1 \wedge T_{O_m}}$ restricted to $\ov{O}_n$ under $\wt{P}_w$, with $w \in \ov{O}_n$, has a bounded density. Since $T_{O_m} \uparrow \infty, P_w$-a.s. for $w \in \ov{O}_n$, one finds that for every Lebesgue negligible subset $N$ of $\ov{O}_n$, one has with $T_{O_m}$ the exit time of $O_m$,
\begin{equation}\label{4.7}
\begin{split}
0  &= \dis\int_{\ov{O}_n} \!\wt{\pi} (dw)  \wt{P}_w [ W_{1 \wedge T_{O_m}} \!\in \!N] \underset{n \r \infty}{\uparrow} \dis\int_{\ov{O}_n} \!\wt{\pi}(dw)  \wt{P}_w [W_1\!\in \!N] \underset{n \r \infty}{\uparrow} 
\dis\int \!\wt{\pi} (dw)  \wt{P}_w [W_1\!\in\!N]
\\[1ex]
&= \wt{\pi}(N).
\end{split}
\end{equation}
Since $N$ is an arbitrary Lebesgue negligible subset of $\ov{O}_n$ and $n \ge 0$ is arbitrary, the claim (\ref{4.6}) follows.
\end{proof}

\n
We will now collect two identities for the stationary distribution $\wt{\pi}$ (or equivalently for the law of $W_0$ under $\wt{P}$). They will be helpful in the next section when proving the condensation bound in Theorem \ref{theo5.1}. They play a similar role as the conservation laws (5.3) and (5.5) on p.~212 of \cite{KuksShir12}, in the context of the stationary stochastically excited Navier-Stokes equation on a $2$-dimensional torus. We introduce the notation (see (\ref{1.5}, (\ref{1.7})):
\begin{equation}\label{4.8}
B_0 = a \Sigma_\ell \,(1 + \delta_\ell) \; \mbox{and} \; B_1 = a \Sigma_\ell \, \lambda_\ell \,(1 + \delta_\lambda).
\end{equation}

\begin{proposition}\label{prop4.4}
Denoting by $\wt{E}$ the $\wt{P}$-expectation, see {\rm (\ref{4.5})}, and with $W_0 = (U_0,V_0)$ the value at time $0$ of the canonical process $(W_t)_{t \ge 0}$, one has
\begin{align}
& B_0 = 2 \wt{E}\, [U_0] \label{4.9}
\\[1ex]
&B_1 = 2 \wt{E}\, [\Sigma_\ell \,\lambda_\ell \, q_\ell(W_0)]. \label{4.10}
\end{align}
\end{proposition}

\begin{proof}
We prove (\ref{4.10}). The proof of (\ref{4.9}) is similar. By Proposition \ref{prop4.2} and stochastic calculus we know that
\begin{equation}\label{4.11}
M_t = U_t - U_0 - \dis\int^t_0 \, B_1 - 2 \Sigma_\ell \, \lambda_\ell \, q_\ell (W_s) ds, t \ge 0,
\end{equation}
is a continuous local martingale with bracket process
\begin{equation}\label{4.12}
\langle M \rangle_t = 4 a \dis\int^t_0 \Sigma_\ell \,\lambda_\ell \, (1 + \delta_\ell)\, q_\ell \,(W_s) \, ds, t \ge 0.
\end{equation}
By stationarity and (\ref{4.2}), $U_t, M_t$, and $\langle M \rangle_t$ are $\wt{P}$-integrable for any $t \ge 0$, and by Doob's Inequality, see \cite{KaraShre88} p.~14, $(M_t)_{t \ge 0}$ is a continuous square integrable martingale. It follows that
\begin{equation}\label{4.13}
0 = \wt{E} [M_1] \stackrel{\rm stationarity}{=} 2 \Sigma_\ell \, \lambda_\ell \, \wt{E} \,[q_\ell \,(W_0)] - B_1.
\end{equation}
This proves (\ref{4.10}). One shows (\ref{4.9}) in a similar fashion with $V_t, t \ge 0$, in place of $U_t, t \ge 0$, for the statement replacing (\ref{4.11}).
\end{proof}

In contrast to the Navier-Stokes case, where the conservation law (5.5) on p.~212 of \cite{KuksShir12}, involves the square of the $L^2$-norm of the gradient of the vorticity, and is hard to make use of, the identities (\ref{4.10}) and (\ref{4.9}), together with Theorem \ref{theo2.3}, will play an important role in the derivation of the condensation bound, see Theorem \ref{theo5.1}.

\medskip
We conclude this section with the

\begin{remark}\label{rem4.5} \rm The probability $\wt{\pi}$ in Theorem \ref{theo4.1} is typically not explicit. However, when all $\delta_\ell$ in (\ref{1.7}) are equal, i.e.~when
\begin{equation}\label{4.14}
\mbox{for some $\delta \in (-1,0], \delta_\ell = \delta$ for all $1 \le \ell \le N$},
\end{equation}
then
\begin{equation}\label{4.15}
\begin{array}{l}
\mbox{$\wt{\pi} = \nu_{a(1 + \delta)}$: the image under the map $(|x|^2, |x|^2_{-1})$ of the Gaussian on $\IR^N$ with}
\\
\mbox{density $\big(\pi a(1 + \delta)\big)^{-N/2} \exp \Big\{- \mbox{\f $\dis\frac{|x|^2}{a(1 + \delta)}$} \Big\} \, dx$ (i.e.~$\mu$ in (\ref{1.6}) when $\delta = 0$).}
\end{array}
\end{equation}

\n
We prove (\ref{4.15}) in the case $\delta = 0$, the general case follows by replacing $a$ with $a(1 + \delta)$. The martingale problem (\ref{3.7}) being well-posed, it suffices by the theorem on p.~2 and Section 4 of \cite{Eche82} to show that for every smooth compactly supported $\psi$ on $\stackrel{_\circ}{C}$ one has
\begin{equation}\label{4.16}
\dis\int_{\stackrel{_\circ}{C}}\,d\nu_a(w) \; \wt{A} \psi (w) = 0.
\end{equation}
To this end note that $\mu$ is the stationary distribution for the diffusion on $\IR^N$ attached to $L = \Sigma_\ell \, \lambda_\ell \, (\frac{a}{2} \, \partial^2_\ell - x_\ell\,\partial_\ell)$, i.e.~$\wt{L}$ in (\ref{1.8}) when all $\delta_\ell$ equal $0$. Hence, setting $g(x) = \psi(|x|^2, |x|^2_{-1})$, one has
\begin{equation}\label{4.17}
\dis\int_{\IR^N} d \mu(x) \, Lg(x) = 0 .
\end{equation}
However, using (\ref{1.10}) to express $Lg$ and the fact that $q_\ell(|x|^2,|x|^2_{-1})$ equals $E^\mu[x^2_\ell \, \big| \, |x|^2, |x|^2_{-1}]$ by (\ref{2.7}), we find by (\ref{1.10}) and (\ref{3.6}) that
\begin{equation}\label{4.18}
0 = \dis\int d\mu(x) \, Lg(x) = \dis\int  d\mu(x) \; \wt{A} \psi ( |x|^2, |x|^2_{-1}) = \dis\int d \nu_a(w) \; \wt{A} \psi(w).
\end{equation}
This proves (\ref{4.16}) and the equality (\ref{4.15}) follows in the case $\delta = 0$. It also follows in the general case (\ref{4.14}), as explained above. \hfill $\square$
\end{remark}

\section{Condensation bound}
\setcounter{equation}{0}

The main object of this section is Theorem \ref{theo5.1}. It states, in the context of the stationary diffusion of the previous section, a quantitative lower bound on the ratio of the expected energy with the expected enstrophy. This inequality takes full significance when the effective diffusion investigated in the present article appears as an inviscid limit in the companion article \cite{SzniWidm26b}. Theorem \ref{theo2.3}, through its upper bound on the functions $\wh{q}_i, i \ge 2$, is an important ingredient in the proof of Theorem \ref{theo5.1}.

\medskip
We recall that $B_0 = a \Sigma_\ell (1 + \delta_\ell)$ and $B_1 = a \Sigma_\ell \, \lambda_\ell (1 + \delta_\ell)$, see (\ref{4.8}). We have
\begin{theorem}\label{theo5.1} (Condensation bound)

\medskip\n
For any $\ell_0$ in $\{3, \dots, N\}$ (recall $N = 2n$), writing $\ell_0 = 2i_0$ or $\ell_0 = 2 i_0 -1$, with $2 \le i_0 \le n$, depending on whether $\ell_0$ is even or odd, one has
\begin{equation}\label{5.1}
2 \wt{E} [U_0 - V_0] \le \mbox{\f $\dis\frac{B_1 - B_0}{\lambda_{\ell_0} - 1}$} + \mbox{\f $\dis\frac{\lambda_3}{\lambda_3- 1}$}  \; \sum^{i_0 - 2}_{k=1} \; \mbox{\f $\dis\frac{1}{n-k}$}  \; B_0 \le  \mbox{\f $\dis\frac{B_1 - B_0}{\lambda_{\ell_0} - 1}$} +  \mbox{\f $\dis\frac{\lambda_3}{\lambda_3-1}$}  \; \mbox{\f $\dis\frac{\ell_0}{N - \ell_0}$} \;B_0 
\end{equation}
(where $\wt{E}$ denotes the $\wt{P}$-expectation, see {\rm (\ref{4.5})}, and the sum in the middle member of {\rm (\ref{5.1})} is omitted when $i_0 = 2$).
\end{theorem}

Note that this inequality shows that the left member of {\rm (\ref{5.1})} is small compared to $B_0 = 2 \wt{E}[U_0]$, see {\rm (\ref{4.9})}, and~there is some ``condensation'', when $\ell_0$ can be chosen so that $\frac{B_1}{B_0} \ll \lambda_{\ell_0}$ and $\ell_0 \ll N$. The ratio $B_1/B_0$ can be viewed as an ``effective spectral value'' at which Brownian forcing occurs. We refer to Remark 6.2 of \cite{SzniWidm26b} for a more detailed discussion.

\begin{proof}
We have by Proposition \ref{prop4.4} and (\ref{2.3})
\begin{equation}\label{5.2}
\begin{split}
B_1 - B_0 = 2 \wt{E} \,\big[\Sigma_{\ell > 2} (\lambda_\ell - 1) \, q_\ell(W_0)] & = 2 \wt{E} \,\big[ \Sigma_{\ell > 2} (\lambda_{\ell \vee \ell_0} -1) \, q_\ell (W_0)\big]
\\[1ex]
& - 2 \wt{E} \,\Big[ \sum_{2 < \ell < \ell_0} (\lambda_{\ell_0} - \lambda_\ell)\, q_\ell (W_0)\Big]\,.
\end{split}
\end{equation}
It then follows that
\begin{equation}\label{5.3}
\begin{array}{l}
2 ( \lambda_{\ell_0} -1 ) \, \wt{E} [ U_0 - V_0] \stackrel{(\ref{2.3})}{=} 2 (\lambda_{\ell_0} - 1) \, \wt{E} \,\Big[ \Sigma_{\ell > 2} \; \Big( 1 - \mbox{\f $\dis\frac{1}{\lambda_\ell}$}\Big) \, q_\ell (W_0)\Big] \le
\\[1ex]
2 \wt{E} \big[ \Sigma_{\ell > 2} (\lambda_{\ell \vee \ell_0} -1) \, q_\ell (W_0)\big]  \stackrel{(\ref{5.2})}{=} B_1 - B_0 + 2 \wt{E}\, \Big[\sum\limits_{2 < \ell < \ell_0} (\lambda_{\ell_0} - \lambda_\ell)\, q_\ell (W_0)\Big]\,.
\end{array}
\end{equation}

\n
The last expectation of (\ref{5.3}) will now be bounded with the help of the important inequality (\ref{2.12}). It is equal to
\begin{equation}\label{5.4}
\begin{array}{l}
\wt{E} \, \Big[\sum\limits_{1 < i < i_0} (\mu_{i_0} - \mu_i) \, \wh{q}_i (W_0)\Big] \stackrel{(\ref{2.12})}{\le} \sum\limits_{1< i < i_0} \; \mbox{\f $\dis\frac{\mu_{i_0} - \mu_i}{n - i + 1}$} \; \mbox{\f $\dis\frac{\wt{E} [U_0 - V_0]}{1 - 1/\mu_i}$} \le
\\[2ex]
(\lambda_{\ell_0}-1) (1 - \mu_2^{-1})^{-1} \, \Sigma^{i_0 -2 }_{k=1} \; \mbox{\f $\dis\frac{1}{n-k}$} \; \wt{E} [U_0 - V_0]
\\[1.5ex]
 \mbox{(the sum is understood as $0$ when $i_0 = 2$)}.
\end{array}
\end{equation}

\n
Inserting this upperbound for the last expectation in (\ref{5.3}), using (\ref{4.9}), and dividing by $(\lambda_{\ell_0} -1)$ yields the first inequality of (\ref{5.1}). The second inequality is immediate since $2(i_0 - 2) \le \ell_0$. This proves Theorem \ref{theo5.1}.
\end{proof}

\begin{appendix}
\section{Appendix: Some density and volume calculations}
\setcounter{equation}{0}

In this appendix we collect various probability density and volume calculations, which are useful in the study of the functions $q_\ell$ of Section 2. Of special interest are informations concerning the $(n-2)$-dimensional volume of the compact convex polytope $\IV_w$ in (\ref{A.9}).

\medskip
The set-up is (\ref{1.1}) - (\ref{1.6}), but in this appendix we only assume that
\begin{equation}\label{A.1}
\mbox{$n \ge 3$ (and hence $N = 2n \ge 6$)}.
\end{equation}
Except for the present Appendix A, we always assume $n \ge 4$ (and $N \ge 8$).

\bigskip
We define the random variables on $\IR^N$ endowed with $\mu$, see (\ref{1.6}):
\begin{equation}\label{A.2}
\begin{split}
U &=  |x|^2, V = |x|^2_{-1}, S_i = x^2_{2i} + x^2_{2 i - 1}, 1 \le i \le n, \; \mbox{so that}
\\[0.5ex]
U &=  \Sigma^n_1 \,S_i \; \mbox{and} \; V = \Sigma^n_1 \; \mbox{\f $\dis\frac{1}{\mu_i}$} \; S_i .
\end{split}
\end{equation}

\smallskip\n
Under the law $\mu$ the $S_i, 1 \le i \le n$, are independent exponential variables with parameter $1/a$. Using the change of variable formula, we see that for $\Phi$ bounded measurable function on $\IR^n$ one has (with $E^\mu$ denoting the $\mu$-expectation):
\begin{equation}\label{A.3}
\begin{array}{l}
E^\mu [ \Phi (U,V,S_3, \dots, S_n)] = E^\mu[\Phi (\Sigma^n_1\,S_i, \Sigma^n_1 \; \mbox{\f $\dis\frac{1}{\mu_i}$} \; S_i, S_3, \dots, S_n)] =
\\[1.5ex]
\dis\int \!\!\!\Phi (u,v,s_3,\dots,s_n)  \,1\{(u,v) \in C\}  \,1\{(s_3, \dots, s_n) \in \IV_{u,v}\} 
\\[1ex]
\mbox{\f $\dis\frac{e^{-u/a}}{a^n(1 - \mu_2^{-1})}$} \,du \, dv \, ds_3 \dots ds_n,
\end{array}
\end{equation}
where $C = \{w = (u,v) \in \IR^2; 0 \le v \le u \le \lambda_N v\}$ as in (\ref{1.11}) and setting
\begin{align}
& \Sigma = \{3, \dots, n\}, \label{A.4}
\\[1ex]
&\IV_{u,v} = \Big\{(s_3, \dots, s_n) \in \IR^\Sigma_+; \; u - \mu_1 v \ge \Sigma^n_3 \, \Big(1- \mbox{\f $\dis\frac{\mu_1}{\mu_i}$}\Big)\,s_i \; \mbox{and}
 \label{A.5}
\\
&\qquad \qquad \qquad\qquad  \qquad \quad u - \mu_2 v \le \Sigma^n_3 \, \Big(1-\mbox{\f $\dis\frac{\mu_2}{\mu_i}$}\Big)\,s_i\Big\}, \nonumber
\end{align}
and where it should be observed that $s_1,s_2$ defined by
\begin{equation}\label{A.6}
\left\{ \begin{split} 
\mu_2 ( 1- 1/\mu_2) \,s_1 &=  -u +\mu_2 v + \Sigma^n_3 \, \Big(1-\mbox{\f $\dis\frac{\mu_2}{\mu_i}$}\Big)\,s_i,
\\
( 1- 1/\mu_2) \,s_2 &= u - \mu_1 v - \Sigma^n_3 \, \Big(1-\mbox{\f $\dis\frac{\mu_1}{\mu_i}$}\Big)\,s_i,
\end{split} \right.
\end{equation}
are non-negative when $(s_3, \dots, s_n) \in \IV_{u,v}$ and determined by
\begin{equation}\label{A.7}
\Sigma^n_1 \, s_i = u \; \mbox{and} \; \Sigma^n_1 \; \mbox{\f $\dis\frac{1}{\mu_i}$}\;s_i = v.
\end{equation}
It is convenient to introduce the linear forms on $\IR^\Sigma$:
\begin{equation}\label{A.8}
\ell_1(\sigma) = \Sigma^n_3 \, \Big(1-\mbox{\f $\dis\frac{\mu_1}{\mu_i}$}\Big)\;s_i, \;\ell_2(\sigma) = \Sigma^n_3 \, \Big(1-\mbox{\f $\dis\frac{\mu_2}{\mu_i}$}\Big)\;s_i, \;\mbox{for} \; \sigma = (s_3, \dots, s_n) \in \IR^\Sigma,
\end{equation}
so that for $w = (u,v)$ in $\IR^2$, the compact convex polytope in $\IR_+^\Sigma$:
\begin{equation}\label{A.9}
\IV_w = \{ \sigma \in \IR^\Sigma_+; \; u - \mu_1 v \ge \ell_1 (\sigma) \; \mbox{and} \; u - \mu_2 v \le \ell_2(\sigma)\}
\end{equation}

\smallskip\n
coincides with $\IV_{u,v}$ in (\ref{A.5}) when $w \in C$, and is empty when $w \notin C$.

\medskip
We write for $w = (u,v)$ in $\IR^2$
\begin{equation}\label{A.10}
\mbox{$V(u,v)$ or sometimes $V_w$ for the $(n-2)$-dimensional volume of $\IV_w \subseteq \IR_+^\Sigma$}.
\end{equation}
Note that the volume does not change if one makes the inequalities strict and assume that the components of $\sigma$ are positive in the right member of (\ref{A.9}). With Fatou's Lemma we thus see that
\begin{equation}\label{A.11}
\mbox{$V_w$ is a non-negative continuous function on $\IR^2$ which vanishes outside $C$}.
\end{equation}
It is also positive on the interior $\stackrel{_\circ}{C}$ of $C$, see (\ref{A.18}) below. 

\medskip
Using scaling, one sees from (\ref{A.9}), (\ref{A.10}) that $V_w$ is an homogeneous function of degree $n-2$:
\begin{equation}\label{A.12}
\mbox{$V_{\lambda w} = \lambda^{n-2} V_w$, for $\lambda \ge 0$ and $w \in \IR^2$}.
\end{equation}
Further, for $w \in C$ and $i \in \Sigma \,(= \{3, \dots, n\})$ one introduces the important values 
\begin{equation}\label{A.13}
t_{1,i} = \mbox{\f $\dis\frac{u - \mu_1 v}{1 - \mu_1/\mu_i}$}, \;\; t_{2,i} = \mbox{\f $\dis\frac{u - \mu_2 v}{1 - \mu_2/\mu_i}$}, \; \mbox{with $w = (u,v)$}.
\end{equation}
They are such that for $i \in \Sigma$ and $e_i$ the coordinate vector in the direction $i$, one has:
\begin{equation}\label{A.14}
\mbox{$\ell_1(t_{1,i} e_i) = u - \mu_1 v$ and $\ell_2(t_{2,i} e_i) = u - \mu_2 v$, for each $i \in \Sigma$},
\end{equation}

\vspace{-4ex}
\begin{equation}\label{A.15}
\begin{array}{l}
\ell_1(t e_i) < u - \mu_1 v, \; \mbox{exactly when $t < t_{1,i}$},
\\[1ex]
\ell_2(t e_i) > u - \mu_2 v, \; \mbox{exactly when $t > t_{2,i}$},
\end{array}
\end{equation}

\vspace{-1ex}\n
and
\begin{equation}\label{A.16}
\begin{split}
t_{1,i} - t_{2,i} & = \mbox{\f $\dis\frac{1}{(1 - \mu_1/\mu_i)(1 - \mu_2/\mu_i)}$} \; \Big\{\Big(1-\mbox{\f $\dis\frac{\mu_2}{\mu_i}$}\Big) (u- \mu_1 v) - \Big(1-\mbox{\f $\dis\frac{\mu_1}{\mu_i}\Big) (u - \mu_2 v)$}\Big\}
\\[1ex]
& =  \mbox{\f $\dis\frac{\mu_2 - \mu_1}{(1 - \mu_1/\mu_i)(1 - \mu_2/\mu_i)}$}  \;(v - u/\mu_i),
\end{split}
\end{equation}
so that
\begin{equation}\label{A.17}
\mbox{$t_{1,i} > t_{2,i}$ exactly when $u < \mu_i v$}.
\end{equation}

\begin{center}
\psfrag{t1}{$t_{1,j} e_j$}
\psfrag{t2}{$t_{2,j} e_j$}
\psfrag{t1i}{$t_{1,i} e_i$}
\psfrag{t2i}{$t_{2,i} e_i$}
\psfrag{s1}{$\sigma_{i,j}$}
\psfrag{l1}{$\ell_1 < u - \mu_1 v$ and $\ell_2 > u - \mu_2 v$}
\includegraphics[width=6cm]{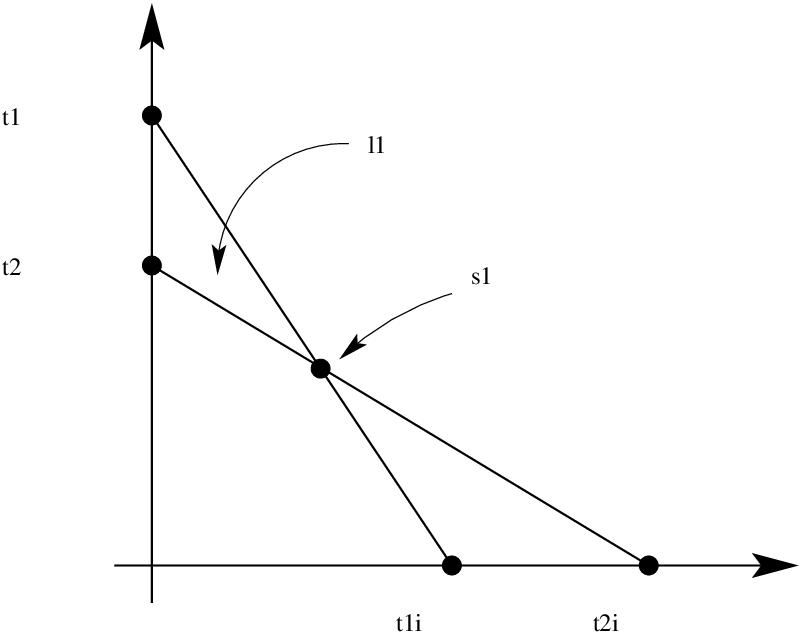}
\end{center}
%
%\smallskip
\begin{center}
Fig.~1: A schematic illustration of $\sigma_{i,j}$ for $i,j$ in $\Sigma$ with $\mu_i < u/v < \mu_j$, see (\ref{A.23}).
\end{center}

\medskip
As we now explain
\begin{equation}\label{A.18}
\mbox{$V_w > 0$ for $w \in \, \stackrel{_\circ}{C}$}.
\end{equation}
Indeed, one has $0 < u < \mu_n v$ and hence $t_{1,n} > t_{2,n} \vee 0 = t^+_{2,n}$ by (\ref{A.17}). For small $\wt{\sigma}$ in $(0, \infty)^\Sigma$, $\frac{1}{2} \;(t_{1,n} + t^+_{2,n}) \; e_n + \wt{\sigma} = \sigma$ belongs to $\IV_w$, see (\ref{A.9}), and  the claim follows. 

\medskip
We now introduce the open subsectors of $C$
\begin{equation}\label{A.19}
\Delta_m = \{w = (u,v) \in \IR^2; \; 0 < \mu_{m-1} \,v < u < \mu_m v\}, \; \mbox{for $2 \le m \le n$},
\end{equation}
so that $\Delta_m$, $2 \le m \le n$, are pairwise disjoint and the union of their closures equals $C$.

\begin{center}
\psfrag{l5}{$e_5$}
\psfrag{l4}{$e_4$}
\psfrag{l3}{$e_3$}
\includegraphics[width=4cm]{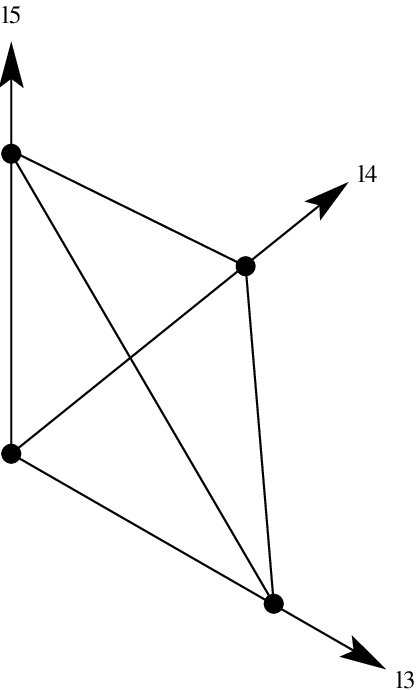} \qquad  \quad \includegraphics[width=4cm]{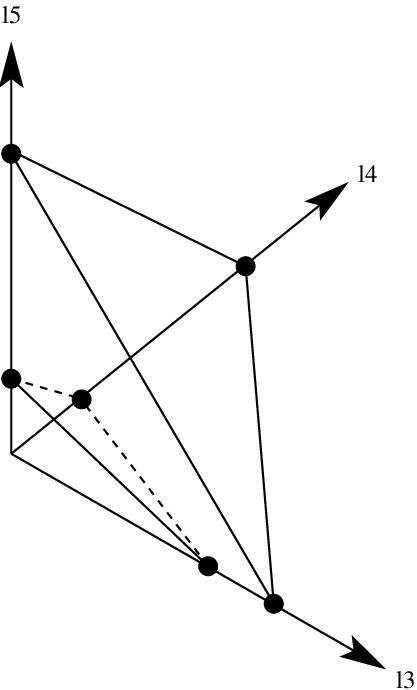} \qquad \quad  \includegraphics[width=4cm]{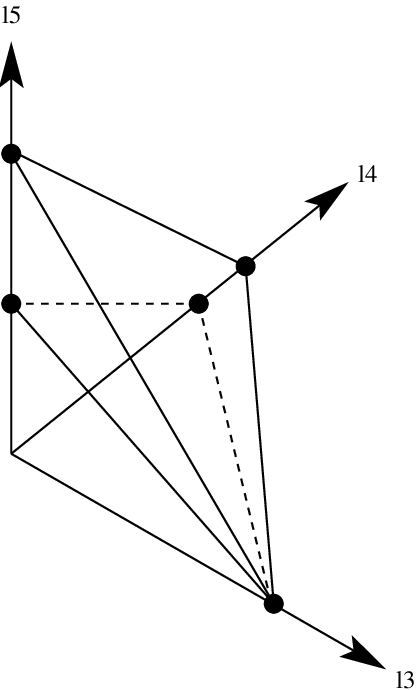}
\end{center}

\qquad $1 = \mu_1 < u/v \le \mu_2$ \qquad \qquad  \qquad $\mu_2 < u/v < \mu_3$  \qquad \qquad \qquad \quad $u/v = \mu_3$

\bigskip
\begin{center}
\psfrag{l5}{$e_5$}
\psfrag{l4}{$e_4$}
\psfrag{l3}{$e_3$}
\psfrag{s35}{$\sigma_{3,5}$}
\psfrag{s34}{$\sigma_{3,4}$}
\psfrag{s45}{$\sigma_{4,5}$}
\includegraphics[width=4cm]{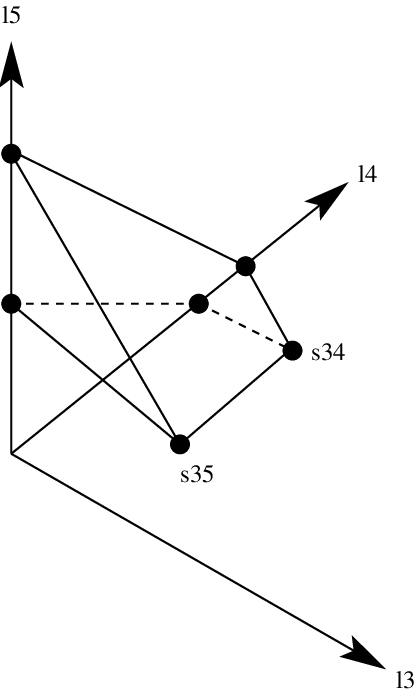} \qquad  \quad \includegraphics[width=4cm]{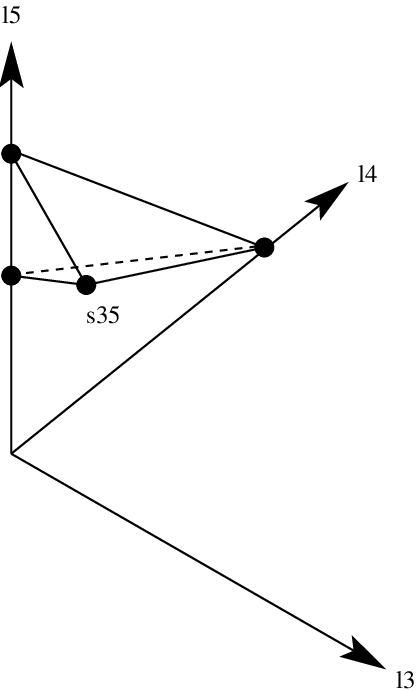} \qquad \quad  \includegraphics[width=4cm]{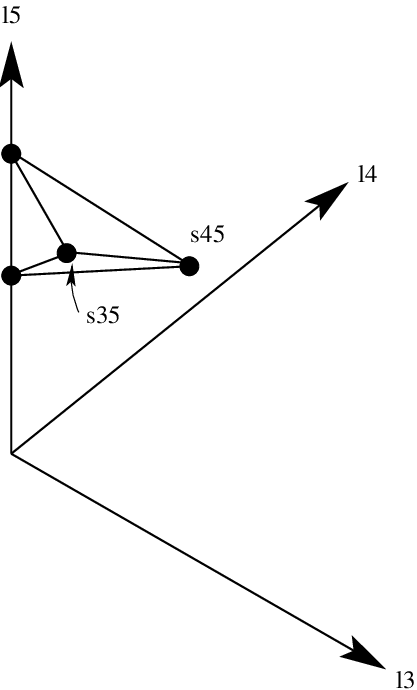}
\end{center}

\qquad $\mu_3 < u/v < \mu_4$ \qquad \qquad  \qquad \qquad  $u/v = \mu_4$  \qquad \qquad \qquad \qquad $\mu_4 < u/v < \mu_5$

\bigskip\bigskip\n
Fig.~2: A schematic illustration of the polytope $\IV_w$ when $n=5$, with varying ratio $u/v$.

\bigskip\bigskip
On $\Delta_2$ and $\Delta_3$ we have the following formulas for the volume of $\IV_w$. 

\medskip
When $w \in \Delta_2$, i.e.~$0 < \mu_1 v(= v) < u < \mu_2 v$, then $u - \mu_2 v < 0$, so that $\IV_w \stackrel{(\ref{A.9})}{=} \{ \sigma \in \IR^\Sigma_+$; $\ell_1(\sigma) \le u - \mu_1 v\}$ (recall $\mu_1 = 1)$ and we find
\begin{equation}\label{A.20}
V_w = \mbox{\f $\dis\frac{1}{(n-2)!}$} \;(u - \mu_1 v)^{n-2} \prod\limits_{3 \le i \le n} \; \Big(1 - \mbox{\f $\dis\frac{\mu_1}{\mu_i}$}\Big)^{-1}.
\end{equation}

\n
When $w \in \Delta_3$, i.e.~$0 < \mu_2 v < u < \mu_3 v$, we see from (\ref{A.16}), (\ref{A.13}) that $t_{1,i} > t_{2,i} > 0$ for all $i$ in $\Sigma$, and we find
\begin{equation}\label{A.21}
V_w = \mbox{\f $\dis\frac{1}{(n-2)!}$} \;\Big\{(u - \mu_1 v)^{n-2} \prod\limits_{3 \le j \le n} \; \Big(1 - \mbox{\f $\dis\frac{\mu_1}{\mu_j}$}\Big)^{-1} - (u-\mu_2 v)^{n-2}  \prod\limits_{3 \le j \le n} \; \Big(1 - \mbox{\f $\dis\frac{\mu_2}{\mu_j}$}\Big)^{-1}\Big\}.
\end{equation}

\bigskip\n
We now turn to the case when $3 < m \le n$. Our main objective is to establish
\begin{proposition}\label{propA.1}
\begin{equation}\label{A.22}
\begin{array}{l}
\mbox{for $3 < m \le n$, the function $V_w$ on $\Delta_m$ coincides with an homogeneous}
\\
\mbox{polynomial of degree $(n-2)$ in $u,v$.}
\end{array}
\end{equation}
\end{proposition}

\begin{proof}
We thus consider a given $m$ as in (\ref{A.22}) and the corresponding $\Delta_m$, see (\ref{A.19}). For $i < m$ and $j \ge m$ in $\Sigma$, we introduce
\begin{equation}\label{A.23}
\begin{array}{l}
\mbox{$\sigma_{i,j} = \alpha_{i,j} e_i + \beta_{j,i} e_i$ with $\alpha_{i,j} > 0$ and $\beta_{j,i} > 0$ given by}
\\[1ex]
\alpha_{i,j} = (v-u/\mu_j) / (\mu_i^{-1} - \mu_j^{-1}), \; \beta_{j,i} = (v - u/\mu_i) \, / \, (\mu^{-1}_j - \mu^{-1}_i)
\\[1ex]
\mbox{(they are both positive since $\mu_{m-1} < u / v < \mu_m$)}.
\end{array}
\end{equation}
Then, one has (see Figure 1)
\begin{equation}\label{A.24}
\ell_1 (\sigma_{i,j}) = u - \mu_1 v \; \mbox{and} \; \ell_2(\sigma_{i,j}) = u - \mu_2 v.
\end{equation}
Since $w \in \Delta_m$, see (\ref{A.19}), $\IV_w$ is a compact convex polytope, which is simple: its extremal vertices saturate exactly $(n-2)$ constraints entering the definition of $\IV_w$ in (\ref{A.9}), and they are
\begin{equation}\label{A.25}
\left\{ \begin{array}{ll}
t_{1,j} \,e_j, j \in \Sigma, j \ge m, \\
\mbox{where the $\ell$-th coordinate vanishes for $\ell \in \Sigma \,\backslash \{j\}$, and}
\\
 \mbox{$\ell_1(\cdot)$ takes the value $u - \mu_1 v$},
\\[1ex]
t_{2,j} \, e_j, j \in \Sigma, j \ge m,  \\
 \mbox{where the $\ell$-th coordinate vanishes for $\ell \in \Sigma \,\backslash \{j\}$, and}
\\
\mbox{$\ell_2(\cdot)$ takes the value $u - \mu_2 v$},
\\[1ex]
\sigma_{i,j}, i < m, j \ge m, \; \mbox{and $i,j \in \Sigma$}, \\
\mbox{where the $\ell$-th coordinate vanishes for $\ell \in \Sigma \,\backslash \{i,j\}$, and}
\\
 \mbox{$\ell_1(\cdot), \ell_2(\cdot)$ respectively take the value $u - \mu_1 v, u - \mu_2 v$}.
\end{array}\right.
\end{equation}
(The other intersections of $(n-2)$ constraints do not lie in $\IV_{w}$.)

\medskip
Having in mind the application of the formula of Lawrence, see \cite{Lawr91}, for the volume $V_w$ of $\IV_w$, we introduce the $(n-2) \times n$ matrix
\begin{equation}\label{A.26}
A =  \begin{pmatrix}
1 - \mu_1/\mu_3 & -(1-\mu_2/\mu_3) & -1
\\
\vdots & \vdots &\quad  \ddots
\\
1 - \mu_1/\mu_i & -(1-\mu_2/\mu_i)&&\!\!\! -1
\\
\vdots & \vdots & &\quad \ddots
\\
1 - \mu_1/\mu_n& -(1-\mu_2/\mu_n)& && \!\!\!-1
\end{pmatrix}
\end{equation}

\medskip\n
so that $\IV_w$ can be realized as the set $(\langle \cdot, \cdot \rangle$ denotes the usual scalar product on $\IR^\Sigma$)
\begin{equation}\label{A.27}
\IV_w = \{x \in \IR^\Sigma; b_j - \langle a_j,x\rangle \ge 0, \; \mbox{for $j = 1,\dots, n\}$},
\end{equation}
with $a_1,\dots,a_n$ the column vectors of $A$ and $b_1,\dots,b_m$ suitable real numbers, see (\ref{A.5}), (\ref{A.9}).

\medskip
We denote by $\cV(\subseteq \IR^\Sigma)$ the collection in (\ref{A.25}). Then for each $\sigma \in \cV$ we denote by
\begin{equation}\label{A.28}
\begin{array}{l}
\mbox{$A_\sigma$ the sub-matrix of $A$ made of columns of $A$ for which}
\\
\mbox{the corresponding inequality defining $\IV_w$ is saturated at $\sigma$}
\\
\mbox{(it is an invertible $(n-2) \times (n-2)$-matrix)}.
\end{array}
\end{equation}
Further, we choose
\begin{equation}\label{A.29}
\begin{array}{l}
\mbox{$f \in \IR^\Sigma$ such that the map $x \in \IR^\Sigma \longrightarrow \langle f,x \rangle$ is non-constant}
\\
\mbox{on each edge of $\IV_w$}.
\end{array}
\end{equation}
Then, the theorem on p.~260 of \cite{Lawr91} shows that
\begin{equation}\label{A.30}
V_w = \mbox{\f $\dis\frac{1}{(n-2)!}$} \; \sum_{\sigma \in \cV} \; \mbox{\f $\dis\frac{\langle f,\sigma \rangle^{n-2}}{{\rm det} \,A_\sigma}$} \; \mbox{\f $\dis\frac{1}{\prod_{i \in \Sigma} \,\langle A^{-1}_\sigma f,  e_i \rangle}$} \,.
\end{equation}
Note that given $w^*$ in $\Delta_m$ (see (\ref{A.19})), $f$ in (\ref{A.29}) can be chosen such that (\ref{A.29}) holds for all $w$ in a neighborhood of $w^*$. Further, with (\ref{A.30}), the dependence of $V_w$ on $w$ in this neighborhood only enters through the term $\langle f, \sigma \rangle^{n-2}$, because the matrix $A$ in (\ref{A.26}) does not depend on $w$. We then note that
\begin{align}
&\mbox{when $\sigma = t_{1,j} \,e_j$, with $j \in \Sigma, j \ge m$, then $\langle f, \sigma \rangle \stackrel{(\ref{A.25}),(\ref{A.13})}{=} f_j (u - \mu_1 v) / (1- \mu_1/\mu_j)$}, \label{A.31}
\\[1ex]
&\mbox{when $\sigma = t_{2,j} \,e_j$, with $j \in \Sigma, j \ge m$, then $\langle f, \sigma \rangle \stackrel{(\ref{A.25}),(\ref{A.13})}{=} f_j (u - \mu_2 v) / (1-\mu_2/\mu_j)$},\label{A.32}
\intertext{and finally,} 
&\mbox{when $\sigma =\sigma_{i,j}$ with $i,j$ in $\Sigma, i < m \le j$, then} \label{A.33}
\\[-0.5ex]
&\langle f,\sigma \rangle \stackrel{(\ref{A.23})}{=} f_i \alpha_{ij} + f_j \,\beta_{ji} = \{f_i (v-u/\mu_j) - f_j(v-u / \mu_i)\} (\mu_i^{-1} - \mu^{-1}_j)^{-1}. \nonumber
\end{align}
Thus, (\ref{A.30}) expressses $V_w$ as a linear combination of $(u-\mu_1 v)^{n-2}$, $(u-\mu_2 v)^{n-2}$, $(v - u/\mu_j)^\ell \,(v-u / \mu_i)^{n-2-\ell}$, $0 \le \ell \le n-2$, with $i < m, j \ge m$ in $\Sigma$, which is locally fixed when $w = (u,v)$ varies in $\Delta_m$. By a patching argument we thus find that $V_w$ restricted to $\Delta_m$ coincides with an homogeneous polynomial in $u,v$ of degree $n-2$. This proves (\ref{A.22}).
\end{proof}

\section{Appendix: Conditional expectations}
\setcounter{equation}{0}

In this appendix we collect calculations, which will provide us with good versions $q_\ell (u,v)$ of the conditional expectations under $\mu$ of $x^2_\ell$  given $|x|^2 = u$ and $|x|^2_{-1} = v$ (see Section 1 for notation). We also establish some identities and properties of these functions. We now assume throughout that
\begin{equation}\label{B.1}
n \ge 4 \; \mbox{(and hence $N = 2n \ge 8$)}.
\end{equation}
We recall that in Appendix A we only assumed $n \ge 3$, see (\ref{A.1}). We first introduce some notation. Given an arbitrary
\begin{equation}\label{B.2}
i_0 \in \Sigma \; (= \{3, \dots,n\}),
\end{equation}
we set
\begin{equation}\label{B.3}
\Sigma^{i_0} = \Sigma \, \backslash \, \{i_0\},
\end{equation}
and for $(u,v) \in C$, as in (\ref{A.5}), (\ref{A.10}), but with the $i_0$-variable omitted, we define
\begin{equation}\label{B.4}
\begin{split}
\IV^{i_0}_{u,v} = \Big\{(s_i)_{i \in \Sigma^{i_0}} \;\mbox{in} \; \IR_+^{\Sigma^{i_0}}; & \; u - \mu_1 v \ge \Sigma_{i \in \Sigma^{i_0}} \,\Big(1 - \mbox{\f $\dis\frac{\mu_1}{\mu_i}$}\Big) \, s_i \; \mbox{and}
\\
&\; u - \mu_2 v \le \Sigma_{i \in \Sigma^{i_0}} \, \Big(1 - \mbox{\f $\dis\frac{\mu_2}{\mu_i}$}\Big) \, s_i \Big\},
\end{split}
\end{equation}
\begin{equation}\label{B.5}
\mbox{$V^{i_0} (u,v)$: the $(n-3)$-dimensional volume of $\IV_{u,v}^{i_0}$}.
\end{equation}
We sometimes write $\IV^{i_0}_w$ and $V^{i_0}_w$ in place of the left members of (\ref{B.4}), (\ref{B.5}) when $w = (u,v) \in C$. We also use the convention
\begin{equation}\label{B.6}
\mbox{$V^{i_0} (u,v) = 0$ when $(u,v) \notin C$}.
\end{equation}
Keeping the variable $s_{i_0}$ in (\ref{A.5}) fixed and integrating the other variables first, we obtain the identity (see (\ref{A.10}) for notation)
\begin{equation}\label{B.7}
V(u,v) = \dis\int^\infty_0 V^{i_0} (u- t, v - t/\mu_{i_0})\,dt, \; \mbox{for $(u,v) \in C$}.
\end{equation}
As in (\ref{A.12}) we note that, see (\ref{B.4}), (\ref{B.5}), $V^{i_0}$ is homogeneous of degree $(n-3)$:
\begin{equation}\label{B.8}
V^{i_0}(\lambda u, \lambda v) = \lambda^{(n-3)} \,V^{i_0}(u,v), \; \mbox{for $(u,v) \in \IR^2, \lambda \ge 0$}.
\end{equation}
We can now introduce on the interior $\stackrel{_\circ}{C}$ of the cone $C$ the function
\begin{equation}\label{B.9}
\wh{q}_{i_0} (u,v) = \mbox{\f $\dis\frac{\dis\int^\infty_0 t \,V^{i_0} (u - t, v - t/\mu_{i_0}) \, dt}{V(u,v)}$}\, .
\end{equation}
Note that the expression under the integral vanishes when $t > u$, and also that $V(u,v) > 0$ since $(u,v) \in \, \stackrel{_\circ}{C}$, (see (\ref{A.18})). we will later see that $\wh{q}_{i_0}(U,V)$ provides a ``good version'' of the conditional expectation $E^\mu [S_{i_0} \, | \, U,V]$, see (\ref{B.23}), and (\ref{A.2}) for notation. We first collect some properties of the functions $\wh{q}_{i_0}$.

\begin{lemma}\label{lemB.1} (With $(u,v) \in \,\stackrel{_\circ}{C}$ and $i_0 \in \Sigma$)
\begin{align}
&\mbox{$\wh{q}_{i_0} (u,v)$ is the expectation of the $i_0$-th coordinate on the set $\IV_{u,v}$ endowed} \label{B.10}
\\
&\mbox{with the normalized Lebesgue measure.} \nonumber
\\[2ex]
&\mbox{$\wh{q}_{i_0} (u,v)$ is a positive, continuous, homogeneous of degree $1$ function on $\stackrel{_\circ}{C}$}. \label{B.11}
\\[2ex]
& \mbox{on each $\Delta_m$, $2 \le m \le n$, $\wh{q}_{i_0}$ coincides with the ratio of an homogeneous} \label{B.12}
\\
& \mbox{polynomial of degree $(n-1)$ in $(u,v)$ with an homogeneous polynomial of} \nonumber
\\
&\mbox{degree $n-2$ in $(u,v)$}. \nonumber
\end{align}
\end{lemma}

\begin{proof}
We begin with (\ref{B.10}). It is a direct consequence of the observation above (\ref{B.7}), of (\ref{A.7}) and (\ref{B.9}). As for (\ref{B.11}), it readily follows from (\ref{A.11}), (\ref{A.18}), and the corresponding statement for $V^{i_0}$, with the domination $V^{i_0} (u,v) \le {\rm const.} \, u^{n-3}$, together with the homogeneity properties (\ref{A.12}), (\ref{B.8}). 

\medskip
We then turn to (\ref{B.12}). Given (\ref{A.22}), we only need to show that
\begin{equation}\label{B.13}
\begin{array}{l}
\mbox{on each $\Delta_m$, $2 \le m \le n$, the numerator in the right member of (\ref{B.9}) is an}
\\
\mbox{homogeneous polynomial of degree $n-1$ in $(u,v)$}.
\end{array}
\end{equation}
To prove (\ref{B.13}) we first note that the ratio
\begin{equation}\label{B.14}
\begin{array}{lll}
\rho(u/v, r) = \mbox{\f $\dis\frac{u/v - r}{1 - r/\mu_{i_0}}$} \; \; \mbox{is} 
\\[2ex]
\mbox{when $u/v = \mu_{i_0}$, identically equal to $\mu_{i_0}$, for $0 < r < \mu_{i_0}$},
\\[1ex]
\mbox{when $u/v > \mu_{i_0}$, an increasing function of $0 < r < \mu_{i_0}$},
\\[1ex]
\mbox{when $u/v < \mu_{i_0}$, a decreasing function of $0 < r < u/v$}.
\end{array}
\end{equation}
Indeed, $\rho(u/v,r) = \frac{u/v - \mu_{i_0}}{1 - r/\mu_{i_0}} + \mu_{i_0}$ on the above ranges.

\medskip
Then, using (\ref{B.8}) and the change of variable $t = vr$ in the numerator of the right member of (\ref{B.9}), we see that the numerator equals
\begin{equation}\label{B.15}
v^{(n-1)} \dis\int^\infty_0 r \, V^{i_0} (u/v - r, 1 - r/\mu_{i_0}) \, dr.
\end{equation}
When $u/v > \mu_{i_0}$, the values $r_i$ where the function $\rho(u/v, \cdot)$ reaches the successive values $\mu_i \ge u/v$ are affine functions of $u/v$. On each interval of $\IR_+$, where $r \not= r_i$, $V^{i_0} (u/v-r, 1 - r/\mu_{i_0})$ is a fixed homogeneous polynomial of degree $(n-3)$ of $u/v-r$, $1 - r/\mu_{i_0}$ by Proposition \ref{propA.1}. In (\ref{B.15}), performing the integration over the various intervals where $r$ reaches the successive values $r_i$, then yields a non-zero polynomial function of $u/v$ with degree at most $n-1$. When $u/v < \mu_{i_0}$, a similar argument can be applied. And when $u/v = \mu_{i_0}$, the integral in (\ref{B.15}) equals $V^{i_0} (\mu_{i_0},1) \int^{\mu_{i_0}}_0 r(1-r/\mu_{i_0})^{n-3} dr$. Multiplying the integral by $v^{(n-1)}$ in all three cases above yields an homogeneous polynomial of degree $n-1$ in $(u,v)$, thus completing the proof of (\ref{B.13}), and hence of (\ref{B.12}).
\end{proof}

Coming back to (\ref{B.9}), (\ref{B.10}) and (\ref{A.3}), we see that for $(u,v) \in \, \stackrel{_\circ}{C}$ 
\begin{equation}\label{B.16}
\begin{array}{l}
\mbox{the vector $(\wh{q}_i(u,v))_{i \in \Sigma}$ is the expectation of $(S_i)_{i \in \Sigma}$ under the conditional}
\\
\mbox{density of $(S_i)_{i \in \Sigma}$ given $U = u$ and $V = v$, and is also the barycenter of $\IV_{u,v}$}
\\
\mbox{(and hence does not depend on $a$ in (\ref{1.5}))}.
\end{array}
\end{equation}
Thus, if we define $\wh{q}_1(u,v)$ and $\wh{q}_2(u,v)$ for $(u,v) \in \, \stackrel{_\circ}{C}$ via (see (\ref{A.6}))
\begin{equation}\label{B.17}
\begin{split}
\mu_2 (1-1/\mu_2) \, \wh{q}_1(u,v) = &\; - u + \mu_2 v + \Sigma^n_3 \,(1- \mu_2/\mu_i) \, \wh{q}_i(u,v)
\\
(1-1/\mu_2) \, \wh{q}_2(u,v) = &\; u - \mu_1 v - \Sigma^n_3 \,(1- \mu_1/\mu_i) \, \wh{q}_i(u,v),
\end{split}
\end{equation}

\medskip\n
then they are positive (because the barycenter of $\IV_{u,v}$ lies in its interior) and
\begin{equation}\label{B.18}
\begin{array}{l}
\mbox{the functions $\wh{q}_i$, $i \in \{1, \dots, n\}$ satisfy (\ref{B.11}), (\ref{B.12}), and do not}
\\
\mbox{depend on the parameter $a$ in (\ref{1.5})}.
\end{array}
\end{equation}
They also satisfy
\begin{equation}\label{B.19}
u = \Sigma^n_{i = 1} \, \wh{q}_i(u,v) \; \mbox{and} \; v = \Sigma^n_1 \; \mbox{\f $\dis\frac{1}{\mu_i}$} \; \wh{q}_i (u,v).
\end{equation}
We now have
\begin{proposition}\label{propB.2}
The functions $\wh{q}_i$, $1 \le i \le n$, can be extended by continuity to $C$, these extensions do not depend on the parameter $a > 0$ in {\rm (\ref{1.5})} and have the boundary values:
\begin{equation}\label{B.20}
\begin{array}{ll}
{\rm a)} & \wh{q}_i (u,v) = 0, \;\mbox{for $u = v$ or $u = \mu_n v$, when $2 \le i \le n-1$, and $u \ge 0$},
\\[1ex]
{\rm b)} & \wh{q}_1 (u,u) = u, \;\wh{q}_1(u,u / \mu_n) = 0, \;\mbox{for $u \ge 0$},
\\[1ex]
{\rm c)} & \wh{q}_n(u,u) = 0, \;\wh{q}_n(u,u / \mu_n) = u, \;\mbox{for $u \ge 0$},
\end{array}
\end{equation}
and they are homogeneous of degree $1$ and satisfy {\rm(\ref{B.19})} on $C$.

\medskip
Moreover, setting (see {\rm (\ref{1.6})}, {\rm (\ref{A.2})} for notation)
\begin{equation}\label{B.21}
q_\ell (u,v) = \fr \; \wh{q}_i (u,v), \; \mbox{for $\ell = 2i$ or $2i-1$, and $(u,v) \in C$},
\end{equation}
one has
\begin{align}
q_\ell (U,V) = &\; \mbox{$E^\mu\,[x_\ell^2 \, | \, U,V], \; \mu$-a.s., for $1 \le \ell \le N$, and} \label{B.22}
\\[1ex]
\wh{q}_i (U,V) = &\; \mbox{$E^\mu\,[S_i \, | \, U,V], \; \mu$-a.s., for $1 \le i \le N$}. \label{B.23}
\end{align}
\end{proposition}

\begin{proof}
To prove the first statement we note that by (\ref{B.19}) on $\stackrel{_\circ}{C}$ one has
\begin{equation}\label{B.24}
u - v = \Sigma^n_1 \,(1-\mu_i^{-1}) \, \wh{q}_i(u,v), \; \mu_n v - u = \Sigma^n_1 \, \Big(\mbox{\f $\dis\frac{\mu_n}{\mu_i}$} - 1\Big) \, \wh{q}_i(u,v).
\end{equation}
Letting $(u,v)$ in $\stackrel{_\circ}{C}$ tend to $(u_0,u_0)$ with $u_0 \ge 0$, we see that the positive $\wh{q}_i (u,v)$ tend to $0$, for $2 \le i \le n$, and by (\ref{B.19}) that $\wh{q}_1(u,v)$ tends to $u$. Similarly, letting $(u,v)$ in $\stackrel{_\circ}{C}$ tend to $(u_0, u_0/\mu_n)$, we find from the second equality of (\ref{B.24}) that $\wh{q}_i (u,v)$ tends to $0$ for $1 \le i \le n-1$, and by (\ref{B.19}) that $\wh{q}_n(u,v)$ tends to $u$. This shows that the $\wh{q}_i$, $1 \le i \le n$, can be extended by continuity to $C$ so that (\ref{B.20}) and (\ref{B.19}) hold. These extended functions are homogeneous of degree $1$ and do not depend on $a$ in (\ref{1.5}), see (\ref{B.12}). As for (\ref{B.23}), it readily follows from (\ref{B.16}), when $3 \le i \le n$, and from (\ref{B.17}) and (\ref{A.6}), when $i = 1$ or $2$. The claim (\ref{B.22}) is then immediate from (\ref{A.2}). This completes the proof of Proposition \ref{propB.2}.
\end{proof}

We refer to Remark A.2 of the companion article \cite{SzniWidm26b} for a sharper version of (\ref{B.22}). We conclude this appendix with some explicit formulas for the $\wh{q}_i$ and $q_\ell$.~These formulas mainly exploit (\ref{B.7}), (\ref{B.9}) as well as the arguments below (\ref{B.14}).

\begin{remark}\label{remB.3}  \rm
\begin{equation}\label{B.25}
\begin{array}{ll}
\mbox{when $3 \le i \le n-1$, $\wh{q}_i\,(\mu_iv, v)= \mu_i \; \mbox{\f $\dis\frac{v}{n-1}$}$, for $v \ge 0$, and}
\\[2ex]
\mbox{when $ 5 \le \ell \le N-2$, $q_\ell\, (\lambda_\ell v, v) = \lambda_\ell \; \mbox{\f $\dis\frac{v}{N-2}$}$, for $v \ge 0$}. 
\end{array}
\end{equation}
By (\ref{B.21}) and (\ref{1.2}), (\ref{1.4}), we only need to prove the identities on the first line of (\ref{B.25}). We can also assume $v=1$ by homogeneity, see below (\ref{B.20}). The argument at the end of the paragraph below (\ref{B.15}), and the fact that $V^i (\mu_i, 1) > 0$, by (\ref{A.18}) and (\ref{B.5}), shows that
\begin{equation}\label{B.26}
\begin{split}
\wh{q}_i \,(\mu_i,1) & = \mbox{\f $\dis\frac{\dis\int^{\mu_i}_0 r(1-r/\mu_i)^{n-3} dr}{\dis\int^{\mu_i}_0 (1-r/\mu_i)^{n-3} \,dr}$} = \mu_i \; \mbox{\f $\dis\frac{\dis\int^{1}_0 t(1-t)^{n-3} dt}{\dis\int^{1}_0 (1-t)^{n-3} \,dt}$} = \mu_i \mbox{\f $\dis\frac{\dis\int^{1}_0 (1-s) \,s^{n-3}\,ds}{\dis\int^{1}_0 s^{n-3} \,ds}$} 
\\[1ex]
& = \mu_i \Big( \mbox{\f $\dis\frac{1}{n-2}$} - \mbox{\f $\dis\frac{1}{n-1}$}\Big) (n-2) = \mbox{\f $\dis\frac{\mu_i}{n-1}$},
\end{split}
\end{equation}
and (\ref{B.25}) follows.

\medskip
Further, one can also compute the values of the $\wh{q}_i$ and $q_\ell$ functions in the closure of the sector $\Delta_2$, see (\ref{A.19}). Namely, for $0 \le v \le u \le \mu_2 v$, one has 
\begin{equation}\label{B.27}
\begin{array}{l}
\wh{q}_i \,(u,v) = \mbox{\f $\dis\frac{1}{n-1}$} \; \mbox{\f $\dis\frac{u-v}{1 - 1/\mu_i}$}, \;\mbox{for} \; 2 \le i \le n, \; \wh{q}_1 (u,v) = v - \mbox{\f $\dis\frac{1}{n-1}$} \; \sum\limits_{i \ge 2} \; \mbox{\f $\dis\frac{u-v}{\mu_i - 1}$}, \quad \mbox{and}
\\[2ex]
q_\ell \,(u,v) = \mbox{\f $\dis\frac{1}{N-2}$} \; \mbox{\f $\dis\frac{u-v}{1 - 1/\lambda_\ell}$}, \;\mbox{for} \; 3 \le \ell \le N, 
\\[2ex]
q_\ell\,(u,v) = \mbox{\f $\dis\frac{v}{2}$}  - \mbox{\f $\dis\frac{1}{2} \dis\frac{1}{N-2}$} \; \sum\limits_{m \ge 2} \; \mbox{\f $\dis\frac{u-v}{\lambda_m - 1}$}, \; \mbox{for $\ell = 1,2$}.
\end{array}
\end{equation}
We only need to prove the identities on the first line, and we can assume that $0 < v < u < \mu_2 v$ (i.e.~$(u,v) \in \Delta_2$). When $i_0 \in \{3, \dots, n\}$, it follows from (\ref{B.9}), (\ref{B.7}) that
\begin{equation}\label{B.28}
\wh{q}_{i_0}\, (u,v) = \mbox{\f $\dis\frac{\dis\int^\infty_0 t \,V^{i_0} (u-t, v-t/\mu_{i_0})\, dt}{\dis\int^\infty_0 V^{i_0} (u-t, v-t/\mu_{i_0})\, dt}$} \;.
\end{equation}
Since $u/v < \mu_2$, the quantity $\rho(u/v,r)$ in (\ref{B.14}) is a decreasing function of $0 < r < u/v$, so that $V^{i_0}(\cdot,\cdot)$ in (\ref{B.28}) solely involves the formula (\ref{A.20}) (with $n-2$ replaced by $n-3$ and $3 \le i \le n$ replaced by $3 \le i \le n$ and $i \not= i_0$). We thus find that the integrals in (\ref{B.28}) run until $t_0$ such that $u - t_0 = v - t/\mu_{i_0}$, i.e.~$t_0\,(1 - 1/\mu_{i_0}) = u-v$. With (\ref{A.20}) we thus find that
\begin{equation}\label{B.29}
\begin{split}
\wh{q}_i \,(u,v) & = \mbox{\f $\dis\frac{\dis\int^{t_0}_0 t(u-t - v + t/\mu_{i_0})^{n-3} dt}{\dis\int^{t_0}_0 (u-t-v+t/\mu_{i_0})^{n-3} \,dt}$} =  \mbox{\f $\dis\frac{\dis\int^{(u-v)/(1-1/\mu_{i_0})}_0 t(u-v-t(1-1/\mu_{i_0}))^{n-3} dt}{\dis\int^{(u-v)/(1-1/\mu_{i_0})}_0 (u-v-t(1-1/\mu_{i_0}))^{n-3} \,dt}$} 
\\
& =  \mbox{\f $\dis\frac{(u-v)}{1-1/\mu_{i_0}}$} \;  \mbox{\f $\dis\frac{\dis\int^1_0 s(1-s)^{n-3} \, ds}{\dis\int^1_0 (1-s)^{n-3} \, ds}$} \stackrel{{\rm (\ref{B.26})}}{=} \mbox{\f $\dis\frac{1}{n-1}$} \;  \mbox{\f $\dis\frac{u-v}{1-1/\mu_{i_0}}$}, \; \mbox{for $3 \le i \le n$}.
\end{split}
\end{equation}
Then, by (\ref{B.17}), we find that
\begin{equation}\label{B.30}
(1 - 1/\mu_2) \, \wh{q}_2\, (u,v) = u-v - \Sigma_{i \ge 3} \,\Big(1 - \mbox{\f $\dis\frac{\mu_1}{\mu_i}$}\Big) \;\wh{q}_i\, (u,v) \stackrel{{\rm (\ref{B.29})}}{=} u-v -  \mbox{\f $\dis\frac{n-2}{n-1}$} (u-v) = \mbox{\f $\dis\frac{u-v}{n-1}$}\,,
\end{equation}
and that
\begin{equation}\label{B.31}
\begin{split}
\wh{q}_1 \,(u,v) & \stackrel{{\rm (\ref{B.19})}}{=} u - \Sigma_{i \ge 2} \,\wh{q}_i \, (u,v) = u - \Sigma_{i \ge 2} \; \mbox{\f $\dis\frac{1}{1 - 1/\mu_i}$} \; \mbox{\f $\dis\frac{u-v}{n-1}$} 
\\[1ex]
&\;\;\,  = v + \Sigma_{i \ge 2} \, \Big(1 - \mbox{\f $\dis\frac{1}{1 - 1/\mu_i}$}\Big) \; \mbox{\f $\dis\frac{u-v}{n-1}$} = v - \mbox{\f $\dis\frac{1}{n-1}$} \; \Sigma_{i \ge 2} \; \mbox{\f $\dis\frac{u-v}{\mu_i -1}$} \,.
\end{split}
\end{equation}
This completes the proof of (\ref{B.27}). \hfill $\square$
\end{remark}
\end{appendix}

\end{document}